\PassOptionsToPackage{bboldPE}{macros4tex}
\documentclass[english,12pt]{article}
\def\stylebib{ext-alphabetic}
\usepackage[T1]{fontenc}
\usepackage[utf8]{inputenc}
\usepackage{amsmath, amsthm}
\usepackage{enumitem}
\usepackage{csquotes}

\makeatletter
\theoremstyle{plain}
\newtheorem{thm}{\protect\theoremname}
\newtheorem{cor}{\protect\corollaryname}
\newtheorem{lem}{\protect\lemmaname}
\newtheorem{s-lem}[lem]{\protect\lemmaname}
\newtheorem{prop}{\protect\propositionname}
\theoremstyle{definition}

\newtheorem{exmp}{\protect\examplename}
\newtheorem*{exmp*}{\protect\examplename}

\ifcsname noxcolor\endcsname
	\usepackage[serif,no-xcolor]{macros4tex}
\else
	\usepackage[serif]{macros4tex}
\fi

\usepackage[unicode=true,pdfusetitle,
	bookmarks=true,bookmarksnumbered=false,bookmarksopen=false,
	breaklinks=false,pdfborder={0 0 0},pdfborderstyle={},backref=false,colorlinks=true]
{hyperref}
\ifcsname noxcolor\endcsname
\else
	\hypersetup{
		urlcolor = URLColor,
		citecolor = CiteColor,
		linkcolor = LinkColor}
\fi

\usepackage[capitalise, noabbrev, nameinlink]{cleveref}

\crefname{thm}{Theorem}{Theorems}
\crefname{lem}{Lemma}{Lemmas}
\crefname{s-lem}{Lemma}{Lemmas}
\crefname{cor}{Corollary}{Corollaries}
\crefname{prop}{Proposition}{Propositions}
\crefname{defn}{Definition}{Definitions}
\crefname{exmp}{Example}{Examples}
\crefname{problem}{Problem}{Problems}
\crefname{figure}{Figure}{Figures}
\crefname{table}{Table}{Tables}
\crefformat{equation}{#2#1#3}

\AtBeginDocument{
	\ifcsname nobiblatex\endcsname
	\else
		\DeclareFieldFormat{titlecase:title}{\MakeSentenceCase*{#1}}
		\DeclareFieldFormat{titlecase:maintitle}{#1}
		\renewbibmacro{in:}{}
		\ifdefined\volnumdelim
			\renewcommand*{\volnumdelim}{\addcolon\space}
		\fi
	\fi
	\renewcommand*{\ref}[1]{\cref{#1}}
}

\makeatother

\usepackage[english]{babel}

\ifcsname nobiblatex\endcsname
\else
	\ifcsname stylebib\endcsname%
	\else%
		\def\stylebib{ext-numeric}
	\fi%

	\usepackage[style=\stylebib, 
		backend=biber, natbib=true, backref=true, giveninits=true, useprefix = true, sorting=nyt, url=false, isbn=false, eprint=true, doi=false, maxalphanames=4, minalphanames=4, minbibnames=6, maxbibnames=99]{biblatex}

	\AtEveryBibitem{
		\ifentrytype{book}{}{
			\clearlist{publisher}
			\clearname{editor}
		}
		\clearfield{note}
	}
	\DeclareFieldFormat{titlecase:title}{\MakeSentenceCase*{#1}}
	\renewbibmacro{in:}{}
\fi

\providecommand{\lemmaname}{Lemma}
\providecommand{\propositionname}{Proposition}
\providecommand{\corollaryname}{Corollary}
\providecommand{\theoremname}{Theorem}
\providecommand{\definitionname}{Definition}
\providecommand{\examplename}{Example}

\usepackage{geometry}
\tcolorboxenvironment{exmp}{breakable, 
enhanced, 
frame hidden,
capture=minipage,
colback=gray!20!white}
\tcolorboxenvironment{thm}{breakable, 
enhanced, 
frame hidden,
capture=minipage,
colback=Dandelion!40!yellow!20!white}
\tcolorboxenvironment{prop}{breakable, 
enhanced, 
frame hidden,
capture=minipage,
colback=Dandelion!40!yellow!20!white}
\tcolorboxenvironment{cor}{breakable, 
enhanced, 
frame hidden,
capture=minipage,
colback=Dandelion!40!yellow!20!white}
\tcolorboxenvironment{s-lem}{breakable, 
enhanced, 
frame hidden,
capture=minipage,
colback=Dandelion!40!yellow!20!white}

\allowdisplaybreaks

\title{Variational Tail Bounds for Norms of Random Vectors and Matrices}

\usepackage{authblk}

\author{Sohail Bahmani} 

\begin{document}
\maketitle
\begin{abstract}
	We propose a variational tail bound for norms of random vectors and matrices under moment assumptions on their one-dimensional marginals. A simplified version of the bound that parametrizes the ``aggregating distribution'' using a certain pushforward of the Gaussian distribution is also provided. We apply the proposed method to reproduce some of the well-known bounds on norms of Gaussian random vectors, and also obtain dimension-free tail bounds for the Euclidean norm of random vectors with arbitrary moment profiles. Furthermore, we reproduce a dimension-free concentration inequality for sum of independent and identically distributed positive semidefinite matrices with sub-exponential marginals, and obtain a concentration inequality for the sample covariance matrix of sub-exponential random vectors. We also obtain a tail bound for the operator norm of a random matrix series whose random coefficients may have arbitrary moment profiles. Furthermore, we use coupling to formulate an abstraction of the proposed approach that applies more broadly. As a corollary, we derive a PAC-Bayesian-style bound in terms of a certain combination of the KL and R\'{e}nyi divergences between the prior and posterior distributions.

\end{abstract}
\section{Introduction}\label{sec:intro}
We revisit the fundamental problem of bounding norms of a random vector $X\in \mbb R^d$. In particular, for a prescribed norm $\norm{\cdot}$ we propose a new variational upper bound to $\norm{X}$ assuming suitable conditions on the moments of the one-dimensional marginals of $X$.

By viewing $\norm{X}$ as a supremum of linear functionals indexed by the unit ball of the dual norm $\norm{\cdot}_*$, variety of techniques that are developed to bound supremum of stochastic processes can be used to bound $\norm{X}$. For example, we can resort to \emph{generic chaining} \citep{Tal14} or its variants to formulate an upper bound for $\norm{X}$. If $X$ has a Gaussian distribution, Talagrand's celebrated \emph{majorizing measures theorem} \citep{Tal87} guarantees that these upper bounds are optimal up to constant factors. However, the bounds obtained through generic chaining are usually abstract in nature and not easy to evaluate since one needs to determine an optimal (or nearly-optimal) \emph{admissible sequence} of the subsets of the unit ball of $\norm{\cdot}_*$. In cases where $X$ is a sum of independent and identically distributed (i.i.d.) random vectors, methods developed in empirical processes theory can be applied. Namely, using the \emph{symmetrization} (see, e.g., \citep*[Lemma 11.4]{BLM13}) approach, the supremum of the corresponding empirical process can be bounded in terms of the \emph{Rademacher complexity} of the set of linear functionals indexed by the unit ball of $\norm{\cdot}_*$ (see \citep*[Ch. 3]{MRT19}, \citep[Ch. 26]{SB14}, and \citep[Ch. 3]{BLM13} for the definition of Rademacher complexity and the related historical remarks).

A distinct alternative to the mentioned methods is a variational approach, commonly known as the PAC-Bayesian argument  \citep{SW97,McA98}, that is primarily developed and applied to analyze the generalization error in statistical learning theory (see, e.g. \citep{LS02,Cat07,AB07,LLS10}). More recently, PAC-Bayesian arguments are used to derive non-vacuous generalization bounds for neural networks that perfectly fit the training samples \citep*{DR17,ZVA+19}. In addition to their application in statistical learning theory, PAC-Bayesian arguments are also used in robust statistics \citep{AC11,Cat12,CG17,Giu18} as well as non-asymptotic random matrix theory \citep{Oli16, Zhi24}.  In its most common formulation (see, e.g., \citep[Theorem 2.1]{Alq24}), to bound $\sup_{\theta \in \Theta} f_\theta(X)$ for a random variable $X$ and given a parameter space $\Theta$, the PAC-Bayesian argument uses the duality between the \emph{cumulant  generating function} and the \emph{Kullback--Leibler (KL) divergence}. Specifically, the variational characterization of the KL divergence \citep[Lemma 2.2]{Alq24} guarantees that
\begin{align}
	\E_{\theta\sim \P}\left(f_\theta(x)\right) & \le \log \E_{\theta\sim\P_{\mr{ref}}}\left(e^{f_\theta(x)}\right) + D_{\mr{KL}}(\P,\P_{\mr{ref}})\,,\label{eq:change-of-measure-KL}
\end{align}
where $D_{\mr{KL}}(\P,\P_{\mr{ref}})$ denotes the KL divergence between the arbitrary ``posterior'' $\P$ and the reference probability measure or  ``prior'' $\P_{\mr{ref}}$ assuming that $\P$ is absolutely continuous with respect to $\P_{\mr{ref}}$. The inequality \eqref{eq:change-of-measure-KL} holds for any $x$, and in particular for $x=X$. Under suitable measurability conditions, $\E_{\theta\sim\P_{\mr{ref}}}\left(e^{f_\theta(X)}\right)$, as a random variable, can be bounded probabilistically using Markov's inequality. It then suffices to choose a suitable reference measure and a suitable class $\mc P$ of aggregating posteriors such that $\sup_{\P\in \mc P}\E_{\theta\sim \P}\left(f_\theta(x)\right)\approx \sup_{\theta\in \Theta}f_\theta (X)$ and $\sup_{\P\in \mc P} D_{\mr{KL}}(\P,\P_{\mr{ref}})$ is small in a certain precise sense. While the PAC-Bayesian arguments are usually based on the change of measure inequality \eqref{eq:change-of-measure-KL}, alternative change of measure inequalities that use other probability metrics and divergences are also considered in the literature \citep{OH21, PG22}.

Some components of the PAC-Bayesian arguments complicate and constrain the application of this method. For example, the fundamental inequality \eqref{eq:change-of-measure-KL} is useful only if $f_\theta(X)$ has a finite exponential moment (in a neighborhood of the origin). More importantly, we need to have an approximation for the KL divergence term in \eqref{eq:change-of-measure-KL} that is both accurate and ``calculation-friendly''. These requirements add to the subtlety of choosing a suitable reference measure $\P_{\mr{ref}}$ and a suitable class of aggregating posteriors $\mc P$. For example, \citet[Theorem 1]{Zhi24} has established a matrix concentration inequality using the PAC-Bayesian argument in which the spectral norm is expressed as the supremum of a bilinear form, and the set $\mc P$ consists of carefully chosen truncated non-isotropic Gaussian distributions for the left and right factors of the bilinear form.

Focusing on tail bounds for the norm of random vectors (or matrices), we propose an alternative variational bound that has simpler structure and applies under fewer constraints compared to the PAC-Bayesian bounds. For example, the standard PAC-Bayesian argument does not apply under the conditions of our results in \ref{thm:general-Euclidean,thm:covariance-matrix-subexponential,thm:matrix-series}. The role of the main components of the proposed bound are also more transparent. Generalizations of the proposed bounds in terms more abstract form of couplings are presented in \ref{sec:coupling-formulation}. In particular, a PAC-Bayesian-style bound is obtained as a corollary in which a certain combination of the KL divergence and the \emph{R\'{e}nyi divergence} is used to measure the closeness of the posterior and prior distributions.
\section{Problem Setup and the Main Lemma}\label{sec:problem-setup+main-lemma}
Let $\norm{\cdot}$ be a norm in $\mbb R^d$ whose dual is denoted by $\norm{\cdot}_*$. Furthermore, denote the canonical centered unit balls of $\norm{\cdot}$ and $\norm{\cdot}_*$ respectively by $\mc B$ and $\mc B_*$. Throughout, we denote the boundary of a set $\mc S$ by $\partial \mc S$. We also use the notation $\norm{u}_M = \sqrt{u^\T M u}$ for any positive semidefinite matrix $M$. We also interpret $p!$ for $p\in \mbb R$ as the standard extension of the factorial to real numbers using the Gamma function. We use the expression $\alpha \lesssim \beta$ to suppress some absolute constant $c>0$ in the inequality $\alpha \le c \beta$.

Let us also introduce some notation that we use to formulate our results. For a probability measure $\P_0$ over $\mbb R^d$ we denote the expectation with respect to $U\sim \P_0$ by $\E_0$, and define
\begin{align*}
	\nu(\P_0) & \defeq \sup_{x\in \partial \mc B} \frac{1}{\P_0\left(|\inp{U,x}|\ge 1\right)}\,,
\end{align*}
which is related to the convex geometric notions of convex cap coverings and convex floating bodies \citep*{BL88, SW90,STTW25}. The central part of our results is the following elementary lemma.
\begin{s-lem}\label{lem:main-lemma}
Given a random variable $X\in \mbb R^d$, for all $u\in \mbb R^d$ and $p\ge 1$ let
\begin{align*}
	M_X(u,p) & \defeq \E\left(\left|\inp{u,X}\right|^p\right)\,.
\end{align*}
Then, we have
\begin{align}
	\E\left(\norm{X}^p\right) & \le \E_0\left(M_X(U,p)\right)\nu(\P_0)\,. \label{eq:main-bound-moments}
\end{align}
Furthermore, for any $t\ge 0$, with probability at least $1-e^{-t}$ we have
\begin{align}
	\norm{X} & \le \inf_{p\ge 1,\,\P_0}        e^{t/p} \left(\E_0\left(M_X(U,p)\right)\right)^{1/p} \nu^{1/p}(\P_0)\,.
	\label{eq:main-bound}
\end{align}
\end{s-lem}
\begin{proof}
	Observe that \eqref{eq:main-bound} follows from \eqref{eq:main-bound-moments} by applying Markov's inequality, taking the $p$-th root, and optimizing with respect to $p\ge 1$ and $\P_0$. So, it suffices to prove \eqref{eq:main-bound-moments}. For $x\in \mbb R^d$ define
	\begin{align*}
		Q_0(x) & \defeq \P_0\left(\left|\inp{U,x}\right|\ge \norm{x}\right)\,,
	\end{align*}
	which is invariant with respect to scaling of $x$, and observe that
	\begin{align*}
		\nu(\P_0) & = \sup_{x\in \partial \mc B} \frac{1}{Q_0(x)}\,.
	\end{align*}
	It follows from Markov's inequality that for any $p\ge 1$ and $x\in \mbb R^d$ we have
	\begin{align}
		\norm{x}^p & \le \E_0\left(\left|\inp{U,x}\right|^p\right)/Q_0(x) \nonumber                    \\
		           & \le \E_0\left(\left|\inp{U,x}\right|^p\right) \nu(\P_0)\,. \nonumber 
	\end{align}
	In particular, for $x=X$ we have
	\begin{align*}
		\norm{X}^p & \le \E_0\left(\left|\inp{U,X}\right|^p\right) \nu(\P_0)\,.
	\end{align*}
	Then, \eqref{eq:main-bound-moments} follows by taking the expectation with respect to $X$, and invoking Tonelli's theorem, whose application is justified because $(u,x)\mapsto |\langle u, x\rangle|^p$ is continuous and thus Borel measurable, i.e.,
	\begin{align*}
		\E\left(\norm{X}^p\right) & \le \E\E_0\left(\left|\inp{U,X}\right|^p\right) \nu(\P_0) = \E_0\left(M_X(U,p)\right)\nu(\P_0)\,.
	\end{align*}
\end{proof}

Let us evaluate the general result of \ref{lem:main-lemma} in some special cases where $M_X(u,p)$ is replaced by some explicit upper bound. Here, we focus on two simple examples for norms of Gaussian random vectors. We emphasize that the lemma can be applied under other suitable moment conditions as well. For example, we may assume that $X$ whose covariance matrix is $\Sigma$, has \emph{sub-exponential} marginals, i.e., for some $\eta > 0$ and all $p\ge 1$, $u\in \mbb R^d$ we have
\begin{align}
	M_X(u,p) & \le p! \left(\eta\,\norm{u}_\Sigma\right)^p\,. \label{eq:sub-exponential}
\end{align}
In fact, in \ref{cor:subGauss-subExp-Euclidean} below we recover, up to constant factors, the results in \citep[Proposition 3]{Zhi24} on the tail bounds for $\norm{X}_2$ under the condition \eqref{eq:sub-exponential}.

\begin{exmp}[Polyhedral norm of a standard Gaussian random vector]\label{exmp:polyhedral-Gaussians}
	Suppose that $X\in\mbb R^d$ is distributed as $\mr{Normal}(0,I)$, for which we have
	\begin{align*}
		M_X(u,2p) \le 2^{p} p! \norm{u}_2^{2p}\,.
	\end{align*}
	We examine the case where $\mc B$, the unit ball of the norm $\norm{\cdot}$, is a symmetric polytope with nonempty interior. In this scenario, $\mc B_*$ is also a polytope and we can write
	\begin{align*}
		\norm{X} & = \max_{u\in \mr{ext}(\mc B_*)} \inp{u,X}\,,
	\end{align*}
	where
	\begin{align*}
		\mr{ext}(\mc B_*) & =\{\pm u_1,\dotsc,\pm u_N\}\,,
	\end{align*}
	is the set of extreme points of $\mc B_*$. Since any norm can be approximated by polyhedral norms to arbitrary precision, this example provides general insights about the nature of the bound \eqref{eq:main-bound}.

	We choose $\P_0$ to be the law of
	\begin{align*}
		U & = \argmax_{u\in \rho_0\,\mr{ext}(\mc B_*)} \inp{u,G}\,,
	\end{align*}
	where $G \sim \mr{Normal}(0,I)$ and $\rho_0>0$ is a parameter to be specified. Let $[N] \defeq \{1,\dotsc,N\}$. For each $k\in[N]$ we define $c_k\ge 1 $ as the smallest real number such that for every $x\in \partial \mc B$, for at least $k$ different $i\in[N]$ we have $c_k\inp{w,u_i} \ge 1$. Formally,
	\begin{align}
		c_k & = \inf\left\{c\ge 1 \st \inf_{x\in \partial \mc B} \sum_{i\in[N]} \bbone(c|\inp{u_i,x}|\ge 1) \ge k\right\}\,.\label{eq:exmp1-c_k}
	\end{align}
	In particular, $c_1 = 1$ due to the fact that $\mc B$ is the polar set of $\mc B_*$. Furthermore, we define
	\begin{align}
		\pi_k & = \min_{I\subseteq[N]: |I|=k} \sum_{i\in I}\P(U=\rho_0 u_i)\,,\label{eq:exmp1-pi_k}
	\end{align}
	which is the smallest probability assigned to any $k$-subset of $\{\rho_0 u_1, \dotsc,\rho_0 u_N\}$ under $\P_0$. Observe that $\pi_k$ does not depend on the scaling parameter $\rho_0$.

	Choosing $\rho_0 = c_k$ and recalling the definition of $\nu(\P_0)$, we have
	\begin{align*}
		\nu(\P_0) & \le \pi_k^{-1}\,.
	\end{align*}
	Therefore, using the fact that $p! \le p^p$, it follows from \ref{lem:main-lemma} that
	\begin{align}
		\norm{X} & \le \inf_{p\ge 1,\,k\in[N]} e^{t/{2p}}\left(\sqrt{2p}\left(\sum_{i\in[N]}2\P_0\left(U=c_k u_i\right)\norm{u_i}_2^{2p}\right)^{1/(2p)} c_k \pi_k^{-1/(2p)}\right)\,,\label{eq:polyhedral-Gaussian}
	\end{align}
	with probability at least $1-e^{-t}$. This bound can be further simplified to
	\begin{align}
		\norm{X} & \le \inf_{p\ge 1,\,k\in[N]} e^{t/{2p}}\left(\sqrt{2p}c_k \pi_k^{-1/(2p)}\right) \max_{i\in[N]}\norm{u_i}_2\nonumber                             \\
		         & \le \sqrt{2e}\max_{i\in [N]}\norm{u_i}_2 \min_{k\in[N]} c_k\max\left\{1,\sqrt{t-\log\pi_k}\right\} \,,\label{eq:polyhedral-Gaussian-simplified}
	\end{align}

	For verification, we can recover, up to the constant factors, the well-known bound for the $\ell_\infty$-norm of $X\sim\mr{Normal}(0,I)$ which states that
	\begin{align*}
		\norm{X}_\infty & \le \sqrt{2\log(2d)} + \sqrt{t}\,,
	\end{align*}
	holds with probability at least $1-e^{-t}$ (see, e.g., \citep[Example 5.7]{BLM13}). This special case is captured in our formulation, by choosing $N=d$ and $u_1,\dotsc,u_d$ to be the canonical unit basis vectors in $\mbb R^d$. Then, due to symmetry, we have $\P_0\left(U = c_k u_1\right)=\dotsb = \P_0\left(U=c_k u_d\right) =  1/(2d)$. Thus
	\begin{align*}
		\sum_{i\in[d]}2\P_0\left(U=c_k u_i\right)\norm{u_i}_2^{2p} & = 1\,,
	\end{align*}
	and choosing $k=1$ we have $c_1 =1$ and $\pi_1 = 1/(2d)$. Therefore, the derived bound \eqref{eq:polyhedral-Gaussian} simplifies to\looseness=-1
	\begin{align*}
		\norm{X}_{\infty} & \le \inf_{p\ge 1} e^{t/{2p}}\left(\sqrt{2p}(2d)^{1/(2p)}\right) \\
		                  & \le \sqrt{2e(\log(2d)+t)}\,.
	\end{align*}
\end{exmp}

\subsection{Dimension-free tail bound for the Euclidean norm of random vectors with arbitrary moment profiles}\label{ssec:general-Euclidean}
Any norm can be approximated by polyhedral norms. However, using such an approximation is not always desirable because it ignores any special structure that the norm $\norm{\cdot}$ might have, and may produce bounds that are not explicit in terms of the parameters that describe the law of the random vector. \ref{lem:main-lemma} can still provide more explicit bounds in some special cases.
We use \ref{lem:main-lemma} to prove the following tail bound for the Euclidean norm of random vectors that have arbitrary ``moment profiles''. We emphasize that the standard PAC-Bayesian argument requires exponential moments to be finite, therefore it cannot, as a standalone method, address the general situation considered by the following theorem.
\begin{thm}\label{thm:general-Euclidean}
	Let $X\in \mbb R^d$ be a centered random vector with covariance $\Sigma$ and a non-decreasing moment profile $h\st{} [1,\infty)\to [0,\infty]$ that for every $u\in \mbb R^d$ and $p\ge 2$ satisfies
	\begin{align}
		\left(\E\left(\left|\inp{u,X}\right|^p\right)\right)^{1/p} & \le h(p) \norm{u}_{\Sigma}\,.\label{eq:moment-profile}
	\end{align}
	Then for any $t>0$, with probability at least $1-e^{-t}$ we have
	\begin{align}
		\norm{X}_2 & \le \inf_{p\ge 2} 2p^{-1/2}h(p) \left(\tr\left(\Sigma\right)+\frac{p}{2}\norm{\Sigma}_{\mr{op}}\right)^{1/2}  e^{(t+\log(2))/p} \,. \label{eq:general-Euclidean-bound}
	\end{align}
\end{thm}
\begin{proof}
	We use \ref{lem:main-lemma} with $\P_0$ being the law of $U = \kappa_0 G$ for a certain $\kappa_0>0$ and a standard Gaussian vector $G\in \mbb R^d$. The assumption on $X$ provides the inequality
	\begin{align*}
		M_X(u,p) & \le h^p(p) \norm{u}^{p}_{\Sigma}\,,
	\end{align*}
	which together with \ref{lem:moments-of-quadratic-form} yields
	\begin{align*}
		\left(\E_0\left(M_X(U,p)\right)\right)^{1/p} & \le  \kappa_0 h(p) \left(\E\left(\norm{G}_\Sigma^{p}\right)\right)^{1/p}                  \\
		                                             & \le \kappa_0 h(p) \left(\tr(\Sigma) + \frac{p}{2} \norm{\Sigma}_{\mr{op}}\right)^{1/2}\,.
	\end{align*}
	Furthermore, recalling that $\mc B$ is the unit $\ell_2$ ball, for any $x\in \partial \mc B$ we can write
	\begin{align*}
		\P_0 \left(\left|\inp{U,x}\right| \ge 1\right) & = \P\left(\kappa_0 |g| \ge 1\right) =  2\Phi(-\frac{1}{\kappa_0})\,,
	\end{align*}
	where $g\in \mbb R$ is a standard Gaussian random variable, and $\Phi(\cdot)$ denotes the standard Gaussian cumulative distribution function (CDF). It follows from straightforward calculus and the inequality $2\pi < 3e$ that $4\Phi(-z)-e^{-2z^2/3}$ is decreasing in $z\ge 0$, which implies that for all $z\ge 0$ we have\footnote{Here we favored simplicity of the lower bound over its accuracy. More accurate lower bounds for the quantile function, or the \emph{Mills ratio} of the standard Gaussian random variables [see e.g., \citep[Exercise 2.2]{Wai19} \citep[Exercise 7.8]{BLM13}, \citep{GU14}, and references therein] can be used instead.}
	\begin{align}
		\Phi(-z) & \ge e^{-2z^2/3}/4\,.\label{eq:Gaussian-CDF-bound}
	\end{align}
	Therefore, for all $x\in \mc B$ we have
	\begin{align*}
		\P_0\left(\left|\inp{U,x}\right| \ge 1\right) & \ge e^{-2/(3\kappa_0^2)}/2\,,
	\end{align*}
	and thus
	\begin{align*}
		\nu\left(\P_0\right) & \le 2e^{2/(3\kappa_0^2)}\,.
	\end{align*}
	It then follows from \ref{lem:main-lemma} (with the constraint $p\ge 2$ rather than $p\ge 1$) that with probability at least $1-e^{-t}$ we have
	\begin{align*}
		\norm{X}_2 & \le \inf_{p\ge 2,\, \kappa_0 >0}  \kappa_0 h(p) e^{2/(3p\kappa_0^2)} \left(\tr\left(\Sigma\right)+\frac{p}{2}\norm{\Sigma}_{\mr{op}}\right)^{1/2} e^{(t+\log(2))/p}\,.
	\end{align*}
	The bound \eqref{eq:general-Euclidean-bound} follows by choosing $\kappa_0 = 2/\sqrt{3p}$ and using the fact that $\sqrt{e/3} < 1 $.
\end{proof}

The following corollary provides more explicit bounds for random vectors that have sub-Gaussian or sub-exponential marginals.
\begin{cor}\label{cor:subGauss-subExp-Euclidean}
	Under the conditions of \ref{thm:general-Euclidean}, if $X$ has sub-Gaussian marginals, in the sense that $h(p) = \eta p^{1/2}$ in \eqref{eq:moment-profile} for some $\eta \ge 1$, then for any $t>0$, with probability at least $1-e^{-t}$ we have
	\begin{align}
		\norm{X}_2 & \le 6\eta \left(\tr\left(\Sigma\right) + t\norm{\Sigma}_{\mr{op}}\right)^{1/2}\,.\label{eq:subGauss-Euclidean}
	\end{align}
	Similarly, if $X$ has sub-exponential marginals, i.e., $h(p) = \eta p$ for some $\eta \ge 1$, then for any $t\ge 1$, with probability at least $1-e^{-t}$ we have
	\begin{align}
		\norm{X}_2 & \le 4\sqrt{e}\eta \left(\sqrt{t\tr\left(\Sigma\right)} + t\norm{\Sigma}_{\mr{op}}^{1/2}\right)\,.\label{eq:subExp-Euclidean}
	\end{align}
\end{cor}
\begin{proof}
	Let $r = \tr(\Sigma)/\norm{\Sigma}_{\mr{op}}$. Applying \ref{thm:general-Euclidean} in the sub-Gaussian case, with probability at least $1-e^{-t}$ we have
	\begin{align*}
		\norm{X}_2 & \le \inf_{p\ge 2} 2\eta \norm{\Sigma}_{\mr{op}}^{1/2}\left(r+\frac{p}{2}\right)^{1/2}  e^{(t+\log(2))/p}\,,
	\end{align*}
	which by choosing $p = 2(t+\log(2) + \sqrt{r(t+\log(2))}) \ge 2(\log(2) + \sqrt{r\log(2)})>2$ yields
	\begin{align*}
		\norm{X}_2 & \le  2\eta \norm{\Sigma}_{\mr{op}}^{1/2}\left(r+t+\log(2) + \sqrt{r(t+\log(2))}\right)^{1/2} \sqrt{e} \\
		           & \le  \sqrt{6e} \eta \norm{\Sigma}_{\mr{op}}^{1/2}\left(r+t+\log(2)\right)^{1/2}                       \\
		           & \le 6\eta  \norm{\Sigma}_{\mr{op}}^{1/2}\left(r+t\right)^{1/2}                                        \\
		           & = 6\eta \left(\tr\left(\Sigma\right) + t\norm{\Sigma}_{\mr{op}}\right)^{1/2}\,,
	\end{align*}
	where in the second line we used $2\sqrt{r(t+\log(2))} \le r+t+\log(2)$, and in the third line we used the facts that $r\ge 1$ and $e(1+\log(2))< 6$.

	In the sub-exponential case, \ref{thm:general-Euclidean} guarantees, with probability at least $1-e^{-t}$, that
	\begin{align*}
		\norm{X}_2 & \le \inf_{p\ge 2} 2\eta \norm{\Sigma}_{\mr{op}}^{1/2}\left(pr+\frac{p^2}{2}\right)^{1/2}  e^{(t+\log(2))/p}\,.
	\end{align*}
	Since $t\ge 1$, choosing $p=2t$ guarantees that $e^{(t+\log(2))/p}\le \sqrt{2e}$ and thus
	\begin{align*}
		\norm{X}_2 & \le 4\sqrt{e}\eta  \norm{\Sigma}_{\mr{op}}^{1/2}\left(tr+t^2\right)^{1/2}                         \\
		           & \le 4\sqrt{e}\eta \left(\sqrt{t\tr\left(\Sigma\right)} + t\norm{\Sigma}_{\mr{op}}^{1/2}\right)\,,
	\end{align*}
	where the second line follows from the sub-additivity of the square root.
\end{proof}

For random vectors with sub-Gaussian marginals, the tail bounds of the form \eqref{eq:subGauss-Euclidean} are well-known and can be derived using more standard methods (see, e.g., \citep[Exercise 6.3.5]{Ver18}). If $X$ is a centered Gaussian, the Gaussian concentration inequality \citep[Theorem 5.5]{BLM13} provides a better approximation for the tail behavior of $\norm{X}_2$ and guarantees that with probability at least $1-e^{-t}$ we have
\begin{align*}
	\norm{X}_2 & \le \sqrt{\tr(\Sigma)} + \sqrt{2\norm{\Sigma}_{\mr{op}}t}\,.
\end{align*}
Therefore, with $h(p) = \sqrt{2p/\pi}$ being a valid moment profile for the Gaussian random vector $X$, \eqref{eq:subGauss-Euclidean} recovers the inequality above up to a factor of $6\sqrt{2/\pi}<5$.

The tail bound \eqref{eq:subExp-Euclidean}, up to the constant factors, is equivalent to the tail bound due to \citet[Proposition 3]{Zhi24} for random vectors with sub-exponential marginals. As pointed out in \citep{Zhi24}, this bound applies to log-concave measures, including uniform measures on convex bodies, and as such it improves on a similar result due to \citet{Ale95}. Also, it can be viewed as a dimension-free specialization of a general result of \citet{LN20} to the case of the Euclidean norm of sub-exponential random vectors.
\section{Simplified Bounds via Pushforward of Gaussians}\label{sec:Gaussian-pushforward}
The tail bound of \ref{lem:main-lemma} holds in rather general situations, but it is not entirely calculation-friendly, mainly due to its implicit form and the minimization with respect to $\P_0$. Assuming that $X$ has sub-Gamma marginals, by characterizing $\P_0$ as a certain pushforward of the standard Gaussian measure, the following theorem provides a bound that can be optimized using a few scalar parameters. In particular, one can recover the tail bound in \ref{exmp:polyhedral-Gaussians}, and \eqref{eq:subGauss-Euclidean} of \ref{cor:subGauss-subExp-Euclidean} using this theorem. These analyses are provided in \ref{apx:examples-via-pushforward}.

\begin{thm}\label{thm:gaussian-pushforward}
	Let $X\in \mbb R^d$ be a zero-mean random variable with covariance matrix $\Sigma \succ 0$ and sub-Gamma marginals in the sense that for some constants $\eta_1,\eta_2\ge 0$, for every $p\ge 1$ and $u\in \mbb{R}^d$ we have
	\begin{align*}
		\E \left(\left|\inp{u, X}\right|^{2p}\right) \le p!\,\left(\eta_1\, \norm{u}_\Sigma\right)^{2p} + (2p)!\,\left(\eta_2\, \norm{u}_*\right)^{2p} \,.
	\end{align*}
	Furthermore, for parameters $\rho_0\ge 0$ and $\kappa_0> 0$, let
	\begin{align*}
		f_0(x) & = \inf_y \frac{\kappa_0}{2}\norm{x- y}_{\Sigma^{-1/2}}^2 + \rho_0 \norm{y}\,,
	\end{align*}
	With $(z)_+ \defeq \max\{z,0\}$ and $G\sim\mr{Normal}(0,I)$ being a standard Gaussian random variable in $\mbb R^d$, let
	\begin{align*}
		\ubar{\lambda}_{\Sigma}(\kappa_0,\rho_0) & = \inf_{x\in \partial \mc B} \E\left(\left(\left|\inp{\nabla f_0(\Sigma^{1/2}G),x}\right|-1\right)_+\right)\,,
	\end{align*}
	and define
	\begin{align*}
		T_{\geqslant}(s)  = 1- T_{<}(s) & = \P\left(\norm{G}_* \ge s\right)\,,
	\end{align*}
	\begin{align*}
		\Omega_{\Sigma}(p,\kappa_0,\rho_0) & = \eta_1 \sqrt{p}\min\Bigg\lbrace \rho_0 \norm{\Sigma}_\square^{1/2},\,\inf_{\tau \in (0,1]}\Bigg[ T^{1/(2p)}_\geqslant\left(\frac{\rho_0}{\tau \kappa_0}\right) \rho_0 \norm{\Sigma}_\square^{1/2} + T_<^{1/(4p)}\left(\frac{\rho_0}{\tau \kappa_0}\right) \\
		                                   & \hspace{3em} \cdot \kappa_0\left( (1-\tau)\norm{\Sigma}_{\mr{op}}^{1/4}\sqrt{\tr(\Sigma^{1/2})} + \tau \sqrt{\tr \Sigma} +\sqrt{2p} \sqrt{\norm{\Sigma}_{\mr{op}}}\right)\Bigg] \Bigg\rbrace + 2\eta_2 p \rho_0\,,
	\end{align*}
	and
	\begin{align*}
		\bar{\nu}_{\Sigma}(\kappa_0, \rho_0) & = \min\left\{\frac{\rho_0}{\ubar{\lambda}_{\Sigma}(\kappa_0,\rho_0)},\,\frac{\kappa_0^2 \left(T_\geqslant(\rho_0/\kappa_0)\left(\norm{\Sigma^{-1}}_{\mr{op}}\norm{\Sigma}_{\mr{op}}-1\right) + 1\right)\rad^2\left(\mc B\right)}{\ubar{\lambda}^2_{\Sigma}(\kappa_0,\rho_0)}+1\right\}\,,
	\end{align*}
	where $\norm{M}_{\square} = \sup_{u\in \mc B_*} u^\T M u$ for any $d\times d$ symmetric matrix $M$, and $\rad(\mc B) = \sup_{x\in \mc B} \norm{x}_2$ denotes the radius of $\mc B$ as measured by the $\ell_2$ norm. Then, for any $t\ge 0$, with probability at least $1-e^{-t}$, we have
	\begin{align}
		\norm{X} & \le \inf \left\{e^{t/(2p)}\,\Omega_{\Sigma}(p,\kappa_0,\rho_0)\,\bar{\nu}^{1/(2p)}_{\Sigma}(\kappa_0,\rho_0)\st{}p\ge 1,\,  \rho_0\ge 0,\, \kappa_0>0\right\}\,.\label{eq:G-pushforward}
	\end{align}
\end{thm}
It is worth mentioning that using duality we can show that the mapping $x\mapsto \nabla f_0(x)$ can be expressed equivalently as
\begin{align}
	\nabla f_0(x) & = \argmin_{u\in \rho_0\mc B_*} \frac{1}{2}\norm{\kappa_0 \Sigma ^{-1/2}x - u}^2_{\Sigma^{1/2}}\,. \label{eq:gradient-as-a-projection}
\end{align}
Specifically, $\nabla f_0(\Sigma^{1/2}G)$ is a projection of $\kappa_0 G$ onto $\rho_0\mc B_*$ in the sense that
\begin{align*}
	\nabla f_0(\Sigma^{1/2}G) & = \argmin_{u\in \rho_0 \mc B_*} \frac{1}{2}\norm{\kappa_0 G - u}^2_{\Sigma^{1/2}}\,.
\end{align*}
This characterization also shows that the law of $U=\nabla f_0(\Sigma^{1/2}G)$ is a mixture of a truncated Gaussian distribution on the interior of $\rho_0\mc B_*$ and another distribution on the boundary of $\rho_0\mc B_*$ that becomes more concentrated around the extreme points as $\kappa_0$ increases (relative to $\rho_0$).
\begin{proof}
	Taking $U=\nabla f_0(\Sigma^{1/2} G)$, it suffices to find the appropriate approximations for the terms that appear in \eqref{eq:main-bound}. The dual representation of $f_0(x)$, i.e.,
	\begin{align*}
		f_0(x) & = \sup_{u\in \rho_0\mc B_*} \inp{u,x} - \frac{1}{2\kappa_0}\norm{u}_{\Sigma^{1/2}}^2\,,
	\end{align*}
	reveals that for all $x\in \mbb R^d$ we have
	\begin{align}
		\norm{\nabla f_0(x)}_* & \le \rho_0\,.\label{eq:bounded-gradient}
	\end{align}

	Furthermore, \ref{lem:gradient-split-bound} below in \ref{apx:supporting-lemmas} guarantees that
	\begin{align*}
		\norm{\nabla f_0(\Sigma^{1/2}G)}_{\Sigma} & \le \norm{\Sigma}_{\mr{op}}^{1/4} (1-\tau)\kappa_0\norm{G}_{\Sigma^{1/2}} + \tau \kappa_0 \norm{G}_\Sigma\,.
	\end{align*}
	Combining the bound above and \eqref{eq:bounded-gradient} depending on whether $\norm{G}_*< \rho_0/(\tau \kappa_0)$ or not, yields
	\begin{align*}
		\norm{\nabla f_0 (\Sigma^{1/2}G)}_\Sigma & \le \bbone\left(\norm{G}_* \ge \rho_0/(\tau \kappa_0)\right)\rho_0 \norm{\Sigma}_\square^{1/2}                                                                                            \\
		                                         & \hspace{3em}+ \bbone\left(\norm{G}_* < \rho_0/(\tau \kappa_0)\right)\left(\norm{\Sigma}_{\mr{op}}^{1/4} (1-\tau)\kappa_0\norm{G}_{\Sigma^{1/2}} + \tau \kappa_0 \norm{G}_\Sigma\right)\,.
	\end{align*}
	Recalling the definitions of the shorthands $T_\geqslant(\cdot)$ and $T_{<}(\cdot)$, we can write
	\begin{align*}
		 & \left(\E\left(\norm{\nabla f_0 (\Sigma^{1/2}G)}_\Sigma^{2p}\right)\right)^{1/(2p)}   \\
		 & \le T_\geqslant^{1/(2p)}\left(\frac{\rho_0}{\tau\kappa_0}\right) \rho_0 \norm{\Sigma}_\square^{1/2} + \\
		 &  T_<^{1/(4p)}\left(\frac{\rho_0}{\tau \kappa_0}\right)\left((1-\tau)\kappa_0\norm{\Sigma}_{\mr{op}}^{1/4} \left(\E\left(\norm{G}_{\Sigma^{1/2}}^{4p}\right)\right)^{1/(4p)} + \tau \kappa_0 \left(\E\left(\norm{G}_\Sigma^{4p}\right)\right)^{1/(4p)}\right)\,, &
	\end{align*}
	where we used the triangle inequality and the inequality $\norm{\xi_1 \xi_2}_{L^{2p}}\le \norm{\xi_1}_{L^{4p}}\norm{\xi_2}_{L^{4p}}$ with $\norm{\cdot}_{L^q} = \left(\E(\left|\cdot\right|^q)\right)^{1/q}$ denoting the $L^q$ norm. Putting the derived bounds together, using \ref{lem:moments-of-quadratic-form}, optimizing with respect to $\tau\in(0,1]$, and recalling the definition of $\Omega_{\Sigma}(p,\kappa_0,\rho_0)$, we obtain
	\begin{align}
		\left(\E_0\left(M_X(\nabla f_0 (\Sigma^{1/2}G),2p)\right)\right)^{1/(2p)} & \le \Omega_{\Sigma}(p,\kappa_0,\rho_0)\,. \label{eq:Omega-bound}
	\end{align}

	Furthermore, using the second moment method (see \ref{lem:PZ} below in \ref{apx:supporting-lemmas}) we have
	\begin{align*}
		\P\left(\left|\inp{\nabla f_0(\Sigma^{1/2}G),x}\right| \ge 1\right) & \ge \frac{\left(\E\left(|\inp{\nabla f_0(\Sigma^{1/2}G),x}| - 1\right)_+\right)^2}{\E\left(\left(|\inp{\nabla f_0(\Sigma^{1/2}G),x}|-1\right)_+^2\right)}\,,
	\end{align*}
	and thus
	\begin{align*}
		\nu(\P_0) & \le \sup_{x\in \partial \mc B} \frac{\E\left(\left(|\inp{\nabla f_0(\Sigma^{1/2}G),x}|-1\right)_+^2\right)}{\left(\E\left(|\inp{\nabla f_0(\Sigma^{1/2}G),x}| - 1\right)_+\right)^2}\,.
	\end{align*}
	To further simplify the right-hand side of the inequality above we consider two cases. In the first case, for $x\in \partial \mc B$ we use the inequality
	\begin{align*}
		\E\left(\left(|\inp{\nabla f_0(\Sigma^{1/2}G),x}|-1\right)_+^2\right) & \le \rho_0\E\left(\left(|\inp{\nabla f_0(\Sigma^{1/2}G),x}|-1\right)_+\right)\,,
	\end{align*}
	which implies
	\begin{align*}
		\nu(\P_0) & \le \sup_{x\in \partial \mc B} \frac{\rho_0}{\left(\E\left(|\inp{\nabla f_0(\Sigma^{1/2}G),x}| - 1\right)_+\right)} \le \frac{\rho_0}{\ubar{\lambda}_{\Sigma}(\kappa_0,\rho_0)}\,.
	\end{align*}
	In the second case, the Gaussian Poincar\'{e} inequality\footnote{We apply the inequality to function $g\mapsto \left(|\inp{\nabla f_0(\Sigma^{1/2}g),x}| - 1\right)_+$ which is weakly differentiable for $\kappa_0<\infty$.} \citep[Theorem 3.20]{BLM13} yields
	\begin{align*}
		\var\left(\left(|\inp{\nabla f_0(\Sigma^{1/2}G),x}| - 1\right)_+\right) & \le \E\left(\norm{\Sigma^{1/2}\nabla^2 f_0(\Sigma^{1/2}G)x}_2^2\right)\,.
	\end{align*}
	The definition of $f_0$ reveals that for a certain convex function $g_0$ we have $f_0 + g_0 = \kappa_0\norm{\cdot}_{\Sigma^{1/2}}^2/2$. It then follows from Alexandrov's differentiability theorem (see, e.g., \citep[Theorem 6.9]{EG25}, \citep[Theorems 14.25 and 14.1]{Vil09}, and \citep{Roc99}) that $\nabla^2 f_0(\Sigma^{1/2}G)\preceq \kappa_0\Sigma^{-1/2}$ almost surely. Particularly, $\nabla^2 f(\Sigma^{1/2}G) = \kappa_0\Sigma^{-1/2}$ whenever $\kappa_0\norm{G}_*< \rho_0$. Therefore, we can write
	\begin{align*}
		 & \var\left(\left(|\inp{\nabla f_0(\Sigma^{1/2}G),x}| - 1\right)_+\right)                                                                                                                                                          \\
		 & \le  \E\left(\bbone\left(\norm{G}_*\ge \rho_0/\kappa_0\right)\norm{\Sigma^{-1}}_{\mr{op}}\norm{\nabla^2 f_0(\Sigma^{1/2}G)}_{\mr{op}}\norm{x}_2^2 + \bbone\left(\norm{G}_* < \rho_0/\kappa_0\right)\kappa_0^2\norm{x}_2^2\right) \\
		 & \le \kappa_0^2 \left(T_\geqslant(\rho_0/\kappa_0)\left(\norm{\Sigma^{-1}}_{\mr{op}}\norm{\Sigma}_{\mr{op}}-1\right) + 1\right)\norm{x}_2^2\,,
	\end{align*}
	which in turn implies
	\begin{align*}
		\nu(\P_0) & \le \frac{\kappa_0^2\left(T_\geqslant(\rho_0/\kappa_0)\left(\norm{\Sigma^{-1}}_{\mr{op}}\norm{\Sigma}_{\mr{op}}-1\right) + 1\right) \rad^2\left(\mc B\right)}{\ubar{\lambda}^2_{\Sigma}(\kappa_0,\rho_0)}+1\,.
	\end{align*}
	Therefore, recalling the definition of $\bar{\nu}_{\Sigma}(\kappa_0,\rho_0)$, we have shown that
	\begin{align}
		\nu(\P_0) & \le \bar{\nu}_{\Sigma}(\kappa_0,\rho_0)\,. \label{eq:nu-bar-bound}
	\end{align}
	Then \eqref{eq:G-pushforward} follows by applying the inequalities \eqref{eq:Omega-bound} and \eqref{eq:nu-bar-bound} on the bound provided by \ref{lem:main-lemma}.
\end{proof}
\section{Application to Random Matrices}\label{sec:random-matrix-examples}
\subsection{A concentration inequality for sum of random positive semidefinite matrices}\label{sec:sum-of-random-PSD-matrices}
We mentioned previously that we may easily adapt \ref{lem:main-lemma} to apply under different moment conditions. Furthermore, the special structure of the norm of interest $\norm{\cdot}$ can be leveraged to adjust the choice of $\P_0$ in \ref{lem:main-lemma}. As an example, we use \ref{lem:main-lemma} to recover, up to the constant factors, a result of \citet[Theorem 1]{Zhi24} that itself generalizes a concentration inequality due to \citet[Theorem 9]{KL17} for sample covariance matrices of sub-Gaussian random vectors.\footnote{The emphasis in \citep{KL17} is on the concentration inequalities in possibly infinite dimensional spaces.}
\begin{thm}\label{thm:matrix-concentration}
	Let $Z_1,\dotsc,Z_n$ be i.i.d. copies of a random positive semidefinite matrix $Z  \in \mbb R^{\ell \times \ell}$ whose mean is $\Sigma = \E Z$ and for some parameter $\eta \ge 1$ and every $p\ge 1$ and $u\in \mbb R^\ell$ obeys
	\begin{align*}
		\left(\E\left(\left|u^\T Z u\right|^p\right)\right)^{1/p} & \le \eta p \norm{u}_\Sigma^2\,.
	\end{align*}
	Denoting the effective rank of $\Sigma$ by
	\begin{align*}
		r_{\mr{eff}}(\Sigma) & \defeq \frac{\tr{\Sigma}}{\norm{\Sigma}_{\mr{op}}}\,,
	\end{align*}
	for any $t\ge 0$, with probability at least $1-e^{-t}$, we have
	\begin{align*}
		\norm{\frac{1}{n}\sum_{i\in[n]}Z_i - \Sigma}_{\mr{op}} & \lesssim \eta \norm{\Sigma}_{\mr{op}} \sqrt{\frac{2r_{\mr{eff}}(\Sigma) + t}{n}}\max\left\{1,\sqrt{\frac{r_{\mr{eff}}(\Sigma) + t/2}{n}}\right\}\,.
	\end{align*}
\end{thm}
\begin{proof}
	Let $X = \sum_{i=1}^n \left(Z_i - \Sigma\right)$, and $\P_0$ be the law of $U= \kappa_0^2 GG^\T$ where $G$ is a standard Gaussian random vector in $\mbb R^\ell$ independent of $(Z_i)_{i\in[n]}$.
	With $\mbb M_{\mr{sym}}^{\ell}$ denoting the set of $\ell\times \ell$ real symmetric matrices, it follows from \ref{lem:main-lemma} that
	\begin{align}
		\norm{X}_{\mr{op}} & \le \sup_{M\in \mbb M_{\mr{sym}}^\ell\st\norm{M}_{\mr{op}}=1}e^{t/p}\kappa_0^2 \left(\E\left(\left|G^\T X G\right|^p\right)\right)^{1/p}\left(\P\left(\kappa_0^2\left|G^\T M G \right|\ge 1\right)\right)^{-1/p}\,. \label{eq:matrix-master}
	\end{align}
	First, we bound $\left(\E\left(\left|G^\T X G\right|^p\right)\right)^{1/p}$. With $\varepsilon_1,\dotsc,\varepsilon_n$ being a sequence of i.i.d. Rademacher random variables independent of everything else, the standard symmetrization argument guarantees
	\begin{align*}
		\left(\E\left(\left|u^\T X u\right|^p\right)\right)^{1/p} & \le 2 \left(\E\left(\left|\sum_{i\in[n]} \varepsilon_i u^\T Z_i u\right|^p\right)\right)^{1/p}\,.
	\end{align*}
	Since $\sum_{i\in[n]} \varepsilon_i u^\T Z_i u$ is a sum of independent symmetric random variables, for $p\ge 2$, it follows from a moment bound due to \citet[Corollary 2]{Lat97} that 
	\begin{align}
		\left(\E\left(\left|u^\T X u\right|^p\right)\right)^{1/p} & \lesssim \sup_{q\in[\max\{2,p/n\},p]}\frac{p}{q}\left(\frac{n}{p}\right)^{1/q}\left(\E\left(\left|u^\T Zu\right|^q\right)\right)^{1/q} \nonumber \\ 
		                                                          & \lesssim \sup_{q\in[\max\{2,p/n\},p]}\eta p\left(\frac{n}{p}\right)^{1/q}\norm{u}_\Sigma^2                                             \nonumber                      \\
		                                                          & \lesssim \eta \max\{\sqrt{np},p\}\,\norm{u}_\Sigma^2\,. \label{eq:marginal-moments-Latala}
	\end{align}
	In fact, by doubling the hidden constant, \eqref{eq:marginal-moments-Latala} holds for all $p\ge 1$, because for $p\in [1,2]$ we have $\left(\E\left(\left|u^\T X u\right|^p\right)\right)^{1/p}\le \left(\E\left(\left|u^\T X u\right|^2\right)\right)^{1/2}$ and $\max\{\sqrt{np},p\} \le \max\{\sqrt{2n},2\}\le 2 \max\{\sqrt{np},p\}$. Therefore, using \eqref{eq:marginal-moments-Latala}, for any $p\ge 1$,  together with \ref{lem:moments-of-quadratic-form} we obtain
	\begin{align}
		\left(\E\left(\left|G^\T X G\right|^p\right)\right)^{1/p} & \lesssim \eta \max\{\sqrt{np},p\}  \left(\E\left(\norm{G}_\Sigma^{2p}\right)\right)^{1/p} \nonumber                                \\
		                                                          & \lesssim \eta \max\{\sqrt{np},p\} \left(\tr\left(\Sigma\right) + p \norm{\Sigma}_{\mr{op}}\right)\,.\label{eq:matrix-moment-bound}
	\end{align}

	Next, we bound $\P\left(\kappa_0^2\left|G^\T M G \right|\ge 1\right)$ for $M \in \mbb M_{\mr{sym}}^\ell$ with $\norm{M}_{\mr{op}} = 1$. We use a simple form of a general technique due to \citet[Lemma 1.2]{Cha19}. Let $\tilde{G} = (I - vv^\T)G + \tilde{g} v$ where $v$ is the singular vector of $M$ corresponding to the singular value $1$, and $\tilde{g}\sim \mr{Normal}(0,1)$ is independent of $G$. Since $G$ and $\tilde{G}$ are both standard Gaussian random vectors, we have
	\begin{align*}
		2\P\left(\kappa_0^2|G^\T M G| \ge 1\right) & = \P\left(\kappa_0^2|G^\T M G| \ge 1\right) + \P_0\left(\kappa_0^2|\tilde{G}^\T M \tilde{G}| \ge 1\right) \\ & \ge \P\left(\kappa_0^2|\tilde{G}^\T M \tilde{G} - G^\T M G| \ge 2\right)        \\
		                                           & = \P\left(\kappa_0^2\left|\inp{\tilde{G},v}^2 - \inp{G,v}^2\right| \ge 2\right)                           \\
		                                           & = \P\left(\kappa_0^2\left|\inp{\tilde{G}+G,v}\right|\left|\inp{\tilde{G}-G,v}\right| \ge 2\right)\,.
	\end{align*}
	Observe that, by construction, $\inp{\tilde{G} + G,v}/\sqrt{2}$ and $\inp{\tilde{G}-G,v}/\sqrt{2}$ are independent standard Gaussian random variables. Therefore, the above inequality implies that
	\begin{align}
		\P\left(\kappa_0^2|G^\T M G| \ge 1\right) & \ge \frac{1}{2}\P\left(\frac{1}{\sqrt{2}}\left|\inp{\tilde{G}+G,v}\right|\ge \frac{1}{\kappa_0}\right)\P\left(\frac{1}{\sqrt{2}}\left|\inp{\tilde{G}-G,v}\right| \ge \frac{1}{\kappa_0}\right) \nonumber \\
		                                          & = \frac{1}{2}\left(2\Phi\left(-\kappa_0^{-1}\right)\right)^2 \nonumber                                                                                                                                   \\
		                                          & \ge \frac{1}{8}e^{-4/(3\kappa_0^2)}\,, \label{eq:matrix-probability-bound}
	\end{align}
	where we used \eqref{eq:Gaussian-CDF-bound} on the third line. Applying the bounds \eqref{eq:matrix-moment-bound} and \eqref{eq:matrix-probability-bound} in \eqref{eq:matrix-master} we obtain
	\begin{align*}
		\norm{X}_{\mr{op}} & \lesssim e^{t/p}\eta \max\{\sqrt{np},p\}\kappa_0^2 \left(\tr(\Sigma) + p\norm{\Sigma}_{\mr{op}}\right)8^{1/p}e^{4/(3\kappa_0^2 p)}                          \\
		                   & = C \eta \norm{\Sigma}_{\mr{op}}e^{(t+\log(8))/p}\sqrt{\max\{n,p\}}\left(p^{-1/2}r_{\mr{eff}}(\Sigma) + p^{1/2}\right)\kappa_0^2 p e^{4/(3\kappa_0^2 p)}\,,
	\end{align*}
	where $C$ represents a positive absolute constant. Minimizing the right-hand side of the inequality with respect to $\kappa_0 >0$, choosing $p = r_{\mr{eff}}(\Sigma) + t/\log(8)$, and using the facts that $r_{\mr{eff}}(\Sigma)\ge 1$ and $\log(8)>2$ we conclude that with probability at least $1-e^{-t}$ we have
	\begin{align*}
		\norm{\frac{1}{n}\sum_{i\in[n]}Z_i - \Sigma}_{\mr{op}} = \frac{1}{n}\norm{X}_{\mr{op}} & \lesssim \eta \norm{\Sigma}_{\mr{op}}\sqrt{\frac{2r_{\mr{eff}}(\Sigma) + t}{n}}\max\left\{1,\sqrt{\frac{r_{\mr{eff}}(\Sigma) + t/2}{n}}\right\}\,. 
	\end{align*}
\end{proof}
\subsection{A concentration inequality for the sample covariance matrix of sub-exponential random vectors}\label{ssec:covariance-subesponential}
\citet[Theorem 3]{Zhi24} also provides a concentration inequality for the sample covariance matrix of centered log-concave distributions. Following \citet{ALPT09}, \citet{Zhi24} approximates the quadratic form induced by the sample covariance matrix by a term obtained through a certain direction-dependent truncation, and a term formed by the ``peaky'' residuals. The PAC-Bayesian argument is then used to obtain the tail bound for the ``truncated'' term, whereas the tail bound for peaky term is obtained using the \emph{decoupling-chaining} method \citep[Proposition 9.4.2]{Tal14} as cited in \citep{Zhi24}, combined with the tail bound for the Euclidean norm of log-concave random vectors due to \citet[Thoerem 1.1]{Pao06} (see also \citep[Theorem 2]{ALL+14}). In fact, Paouris's inequality is crucial in the proof of \citet[Theorem 3]{Zhi24} in the sense that it is tailored to log-concave measures instead of merely using the sub-exponential property.\footnote{Log-concave measures are ``regular'' (see, e.g., \citep[Definition 3.4]{LW08}, \citep[the statement following Definition 2]{Lat17}, \citep{Mur22}) which also implies that they are sub-exponential.}

Our next result, \ref{thm:covariance-matrix-subexponential}, provides a concentration inequality for the sample covariance matrix of random vectors with sub-exponential distribution. Similar to the model considered in \ref{thm:general-Euclidean}, the marginals of the sample covariance matrix do not have finite exponential moments, which rules out the analysis based solely on the PAC-Bayesian argument. While not adapted for log-concave distributions, without resorting to truncation techniques and relying only on \ref{lem:main-lemma}, we can partially recover the result of \citet[Theorem 3]{Zhi24} with a matching confidence level. Specifically, under the extra condition that the number of samples is at least the cube of the effective rank of the population covariance matrix, the tail bound we provide below in \ref{thm:covariance-matrix-subexponential} is equivalent to the tail bound provided by \citet[Theorem 3]{Zhi24}.

\begin{thm}\label{thm:covariance-matrix-subexponential}
	Let $Z_1,\dotsc,Z_n$ be independent copies of a centered random vector $Z\in \mbb R^d$ with covariance matrix $\Sigma = \E\left(ZZ^\T\right)$ and sub-exponential marginals, i.e., for a certain parameter $\eta \ge 1$ and every $u\in \mbb R^d$ and $p\ge 1$ we have
	\begin{align}
		\left(\E\left(\left|\inp{u,Z}\right|^p\right)\right)^{1/p} & \le \eta p \norm{u}_\Sigma\,. \label{eq:subexponential-marginals}
	\end{align}

	Then, with $r = r_\mr{eff}(\Sigma) = \tr\left(\Sigma\right)/\norm{\Sigma}_{\mr{op}}$ denoting the effective rank of $\Sigma$, for any $t\ge 0$, with probability at least $1-e^{-t}$ we have
	\begin{align*}
		\norm{\frac{1}{n}\sum_{i=1}^n Z_iZ_i^\T - \Sigma}_{\mr{op}} & \lesssim  \begin{cases}
			                                                                        \eta^2 \norm{\Sigma}_{\mr{op}} \max\left\{\frac{\left(r+t+\log(8)\right)^2}{n},\,\left(\frac{r+t+\log(8)}{n}\right)^{1/2}\right\}\,,        & \text{if } r \le n^{1/3}  \\
			                                                                        \eta^2 \norm{\Sigma}_{\mr{op}} \max\left\{\frac{\left(n^{1/3}+t+\log(8)\right)^2}{n},\,\frac{r\left(n^{1/3}+t+\log(8)\right)}{n}\right\}\,, & \text{if } r > n^{1/3}\,.
		                                                                        \end{cases}
	\end{align*}
\end{thm}
\begin{proof}
	The proof is similar to the proof of \ref{thm:matrix-concentration}. Let $X=\sum_{i=1}^n \left(Z_i Z_i^\T -\Sigma\right)$ and $U = \kappa_0^2 GG^\T$ for a parameter $\kappa_0>0$ and a standard Gaussian random vector $G\in \mbb R^d$ that is independent of everything else. We can apply \ref{lem:main-lemma} as in the proof of \ref{thm:matrix-concentration}; modification is only required to approximate the moments of $\left|\inp{U,X}\right|$. Using the standard symmetrization argument, for any $u\in \mbb R^d$ we have
	\begin{align*}
		\left(\E\left(\left|u^\T X u\right|^p\right)\right)^{1/p} & = \left(\E\left(\left|\sum_{i\in [n]}\left(\inp{u,Z_i}^2 - \norm{u}_\Sigma^2\right)\right|^p\right)\right)^{1/p} \\
		                                                          & \le 2 \left(\E\left(\left|\sum_{i\in [n]}\varepsilon_i \inp{u,Z_i}^2\right|^p\right)\right)^{1/p}\,,
	\end{align*}
	where $\varepsilon_1,\dotsc,\varepsilon_n$ are Rademacher random variables independent of each other and everything else. It then follows from \citep[Corollary 2]{Lat97} that for $p\ge 2$ we have
	\begin{align}
		\left(\E\left(\left|\sum_{i\in [n]}\varepsilon_i \inp{u,Z_i}^2\right|^p\right)\right)^{1/p} & \lesssim \sup_{q\in [\max\{2,p/n\},p]}\frac{p}{q}\left(\frac{n}{p}\right)^{1/q}\left(\E\left(\left|\inp{u,Z}\right|^{2q}\right)\right)^{1/q} \nonumber \\
		                                                                                            & \lesssim C_{n,p}\sup_{q\in [\max\{2,p/n\},p]}q^{-2}\left(\E\left(\left|\inp{u,Z}\right|^{2q}\right)\right)^{1/q} \,,\nonumber                          
	\end{align}
	where
	\begin{align*}
		C_{n,p} & \defeq \sup_{q\in[\max\{2,p/n\},p]}pq(n/p)^{1/q}\,.
	\end{align*}
	Using the assumption \eqref{eq:subexponential-marginals}, we obtain
	\begin{align*}
		\left(\E\left(\left|u^\T X u\right|^p\right)\right)^{1/p} & \lesssim  \eta^2 C_{n,p} \norm{u}_\Sigma^2\,.
	\end{align*}
	The function $q\mapsto pq\left(n/p\right)^{1/q}$ has at most one stationary point over $q>0$, which, if exists, is a minimum. Therefore, its supremum with respect to $q\in [\max\{2,p/n\},p]$ is attained at one of the endpoints of the interval. For $p>n$ we have $(p/n)(n/p)^{n/p} < p(n/p)^{1/p}<p$. Consequently, to approximate $C_{n,p}$ it suffices to consider $q=2$ and $q=p$ as the maximizing endpoints, which yields $C_{n,p}\lesssim \max\{2\sqrt{np},\,p^2\}$, and thus
	\begin{align*}
		\left(\E\left(\left|u^\T X u\right|^p\right)\right)^{1/p} & \lesssim \eta^2\max\left\{\sqrt{np},\,p^2\right\}\norm{u}_\Sigma^2 \,.
	\end{align*}

	It then follows from \ref{lem:moments-of-quadratic-form} that
	\begin{align*}
		\left(\E\left(\left|G^\T X G\right|^p\right)\right)^{1/p} & \lesssim \eta^2 \norm{\Sigma}_{\mr{op}} \max\{\sqrt{np},p^2\}\left(r+p\right)                \\
		                                                          & \lesssim \eta^2 \norm{\Sigma}_{\mr{op}} \max\{rn^{1/2}p^{1/2}, n^{1/2}p^{3/2},rp^2, p^3\}\,.
	\end{align*}

	We can again show \eqref{eq:matrix-probability-bound} still holds following the same argument used in the proof of \ref{thm:matrix-concentration}. Therefore, we have
	\begin{align*}
		\norm{X}_{\mr{op}} & \lesssim \eta^2 \norm{\Sigma}_{\mr{op}}  \kappa_0^2 e^{4/(3\kappa_0^2 p)} e^{(t+\log(8))/p}  \max\{rn^{1/2}p^{1/2}, n^{1/2}p^{3/2},rp^2, p^3\}
	\end{align*}
	which by optimizing with respect to $\kappa_0 > 0$ implies
	\begin{align*}
		\norm{X}_{\mr{op}} & \lesssim \eta^2 \norm{\Sigma}_{\mr{op}}  e^{(t+\log(8))/p} \max\{rn^{1/2}p^{-1/2}, n^{1/2}p^{1/2},rp, p^2\}\,.
	\end{align*}
	We consider two cases based on a comparison of $r$ and $n$. In the first case we have $r\le n^{1/3}$, which implies $r \le r^{2/5} n^{1/5} \le n^{1/3} \le r^{-2}n$. It is then straightforward to verify that
	\begin{align*}
		\max\{rn^{1/2}p^{-1/2}, n^{1/2}p^{1/2},rp, p^2\} & = \begin{cases}
			                                                     p^2\,,              & \text{for } p \ge n^{1/3}      \\
			                                                     n^{1/2}p^{1/2}\,,   & \text{for } r\le p \le n^{1/3} \\
			                                                     rn^{1/2}p^{-1/2}\,, & \text{for } 2\le p \le r\,.
		                                                     \end{cases}
	\end{align*}
	Therefore, choosing $p = r+t+\log(8)$ we have
	\begin{align}
		\norm{X}_{\mr{op}} & \lesssim \eta^2 \norm{\Sigma}_{\mr{op}} \max\left\{\left(r+t+\log(8)\right)^2,\,n^{1/2}\left(r+t+\log(8)\right)^{1/2}\right\}\,. \label{eq:bound-n>r^3}
	\end{align}
	Similarly, in the second case we have $r>n^{1/3}$, which implies $r > r^{2/5}n^{1/5} > n^{1/3} > r^{-2}n$. Then we can verify that
	\begin{align*}
		\max\{rn^{1/2}p^{-1/2}, n^{1/2}p^{1/2},rp, p^2\} & = \begin{cases}
			                                                     p^2\,,              & \text{for } p \ge r            \\
			                                                     rp\,,               & \text{for } n^{1/3}\le p \le r \\
			                                                     rn^{1/2}p^{-1/2}\,, & \text{for } p \le n^{1/3}\,.
		                                                     \end{cases}
	\end{align*}
	Choosing $p = n^{1/3}+t+\log(8)$ again we have
	\begin{align}
		\norm{X}_{\mr{op}} & \lesssim \eta^2 \norm{\Sigma}_{\mr{op}} \max\left\{\left(n^{1/3}+t+\log(8)\right)^2,\,r\left(n^{1/3}+t+\log(8)\right)\right\}\,. \label{eq:bound-n<r^3}
	\end{align}
	Recalling that $n^{-1}\sum_{i=1}^n Z_iZ_i^\T - \Sigma  = n^{-1}X$, the result follows from \eqref{eq:bound-n>r^3} and \eqref{eq:bound-n<r^3}.
\end{proof}
\subsection{A tail bound for the operator norm of random matrix series}\label{ssec:matrix-series}
A classical model studied in random matrix theory is the random matrix series model  
\begin{align}
    X & = \sum_{i=1}^n \xi_i A_i\,. \nonumber 
\end{align}
where the matrices $A_1,\dotsc,A_n \in \mbb R^{d_1\times d_2}$ are deterministic and the coefficients $\xi_1,\dotsc,\xi_n$ are random variables. In particular, if $\xi_i$s are i.i.d. standard Gaussian (or Rademacher) random variables the \emph{non-commutative Khintchine inequality} \citep{Lus86,LP91} provides a bound on the \emph{Schatten}-$q$ norms, including the operator norm, of $X$; in its prevalent formulation for $d\times d$ symmetric matrices $A_i$ (see, e.g., \citep[Corollary 2.4]{Tro18}) it guarantees that
\begin{align*}
    \E\norm{X}_{\mr{op}} & \lesssim \sqrt{\log d}\,\norm{\sum_{i=1}^n A_i^2}_{\mr{op}}^{1/2}\,.
\end{align*}
The suboptimal $O(\sqrt{\log(d)})$ multiplicative factor is eliminated recently in \citep*{BBvH23,BCSvH25} which, borrowing ideas from \emph{free probability}, derived a more refined inequality by showing that the spectrum of $X$ is closely approximated by the spectrum of a certain corresponding infinite-dimensional deterministic operator $X_{\mr{free}}$. In particular, \citet[Theorem 1.1]{BCSvH25} showed that
\begin{align}
    \left|\E\left(\norm{X}_{\mr{op}}\right) - \norm{X_{\mr{free}}}\right| & \lesssim \tilde{\upsilon}(X)(\log(d))^{3/4}\,, \label{eq:Gaussian-series-sharp}
\end{align}
where $\norm{X_{\mr{free}}} \le 2\norm{\sum_{i\in [n]} A_i^2}_{\mr{op}}^{1/2}$ and $\tilde{\upsilon}(X) = \norm{\sum_{i\in [n]} A_i^2}_{\mr{op}}^{1/4}\sup_{w\in \mbb R^n\st{}\norm{w}_2\le 1} \norm{\sum_{i=1}^n w_i A_i}_{\F}^{1/2}$. If $\norm{\sum_{i\in [n]}A_i^2}_{\mr{op}}^{1/2}\gg (\log(d))^{\beta}\sup_{w\in \mbb R^n\st{}\norm{w}_2\le 1} \norm{\sum_{i=1}^n w_i A_i}_{\F}$, for $\beta \ge 1$ the bound \eqref{eq:Gaussian-series-sharp} dominates the non-commutative Khintchine inequality, and further becomes dimension-free for $\beta \ge 3/2$. \citet{BvH24} further showed the \emph{universality} of these inequalities by considering $X$ to be a sum of independent random matrices and relating its spectrum to the spectrum of a Gaussian matrix whose mean and covariance operator match those of $X$. 

The following theorem considers a more general model that accommodates the cases where the coefficients $(\xi_i)_{i\in [n]}$ have heavier tail and perhaps are not even independent.

\begin{thm}\label{thm:matrix-series}
    Let $\xi \in \mbb R^n$ be a centered isotropic random vector such that for all $v\in \mbb R^n$ we have
    \begin{align}
        \left(\E\left(\left|\inp{\xi,v}\right|^p\right)\right)^{1/p} & \le h(p) \left(\E\left(|g|^p\right)\right)^{1/p} \norm{v}_2\,, \label{eq:xi-relative-profile}
    \end{align}
    where $g\sim \mr{Normal}(0,1)$ and the function $h:[1,\infty) \mapsto (0,\infty]$ is a moment profile of $\xi$ relative to the standard Gaussian. Furthermore, given matrices $A_1,\dotsc,A_n \in \mbb R^{d_1\times d_2}$ let
    \begin{align*}
        X & = \sum_{i=1}^n \xi_i A_i\,.
    \end{align*}
    For any $t\ge 0$, with probability at least $1-e^{-t}$ we have
    \begin{align}
        \norm{X}_{\mr{op}} & \lesssim \inf_{p\ge 2} e^{t/p} h(p) \left(\sigma_* p^{1/2} + \sigma + \upsilon + \sigma_{\diamond} p^{-1/2}\right)\,, \label{eq:matrix-series-tail-bound}
    \end{align}
    where
    \begin{align*}
        \sigma_*          & = \sup_{w\in \mbb R^n\st{}\norm{w}_2\le 1} \norm{\sum_{i=1}^n w_i A_i}_{\mr{op}}\,,              & \upsilon & = \sup_{w\in \mbb R^n\st{}\norm{w}_2\le 1} \norm{\sum_{i=1}^n w_i A_i}_{\F}\,, \\
        \sigma            & = \norm{\sum_{i=1}^n A_i^\T A_i}_{\mr{op}}^{1/2}+ \norm{\sum_{i=1}^n A_i A_i^\T}_{\mr{op}}^{1/2} &
        \sigma_{\diamond} & = \left(\sum_{i=1}^n \norm{A_i}_\F^2\right)^{1/2}\,.
    \end{align*}
\end{thm}

In the special case of symmetric Gaussian series matrices (i.e., $\xi_i\overset{\mr{i.i.d.}}{\sim}\mr{Normal}(0,1)$) we have $h(p) = 1$, and choosing $p = 2t+2\sigma_{\diamond} /\sigma_* \ge 2$ in \ref{thm:matrix-series} to obtain $\norm{X}_{\mr{op}} \lesssim  \sigma + \upsilon +  \sigma_{\diamond}^{1/2}\sigma_*^{1/2} + \sqrt{2t}\sigma_*$ with probability at least $1-e^{-t}$. The bound
\begin{align}
    \E\norm{X}_{\mr{op}}  & \lesssim  \sigma + \upsilon +  \sigma_{\diamond}^{1/2}\sigma_*^{1/2} \nonumber 
\end{align}
cab be extracted by integrating the tail bound (or directly applying the argument used to prove \ref{lem:main-lemma}). Using the PAC-Bayesian argument, \citet[Theorem 1.2]{Ade25} obtained a bound for the expected \emph{injective norm}\footnote{For a tensor $T\in \mbb R^{d_1\times\dotsb\times d_k}$ the injective norm of $T$ induced by the $q$-norm is $\norm{T}_{\vee_q}=\sup\{\langle u_1\otimes\dotsb\otimes u_k, T\rangle\st{} \forall i\in[k]\ u_i\in \mbb R^{d_i},\,\norm{u_i}_q \le 1\}$} of sub-Gaussian random tensor series (where $A_i$s are tensors and $\xi_i$s are sub-Gaussian). For matrices, the bound \eqref{eq:matrix-series-tail-bound} and that of \citep{Ade25} are the same up to constant factors. These bounds do not have the sub-optimal $\sqrt{\log d}$ multiplicative factor of the non-commutative Khintchine inequality, but they are generally weaker than the more refined inequality \eqref{eq:Gaussian-series-sharp}. For example, for series with symmetric matrices we may have
\begin{align*}
    \sigma_\diamond \sigma_* = \left(\sum_{i=1}^n \norm{A_i}_{\F}^2\right)^{1/2}\sup_{w\st\norm{w}_2\le 1}\norm{\sum_{i=1}^n w_i A_i}_{\mr{op}} & \gg \norm{\sum_{i=1}^n A_i^2}_{\mr{op}}^{1/2} \sup_{w\st\norm{w}_2\le 1}\norm{\sum_{i=1}^n w_i A_i}_{\F} = \frac{\sigma \upsilon}{2}\,,
\end{align*}
where the left-hand side can be as large as $\sqrt{d}$ times the right-hand side (e.g., for $n=d$ and $A_i= e_i\otimes e_i$ for all $i\in[n]$).

The advantage of \ref{thm:matrix-series} is that it addresses the situations where $\xi$ can have a general (relative) moment profile and does not impose the independence condition on the coordinates of $\xi$. For example, if $\xi$ is sub-exponential, even with independent coordinates, the standard PAC-Bayesian argument, as used in \citep{Ade25}, cannot be employed since the analysis involves random variables that do not have finite exponential moments.

\begin{proof}[Proof of \ref{thm:matrix-series}]
    For $i\in [2]$, let $G_i\in \mbb R^{d_i}$ be independent standard Gaussian random vectors. Furthermore, for parameter $\kappa_0>0$, let $\P_0$ to be the law of
    \begin{align*}
        U & = \kappa_0 G_1\otimes G_2\,.
    \end{align*}
    To apply \ref{lem:main-lemma} we first find an upper bound for $\left(\E_0\left(M_X(U,p)\right)\right)^{1/p}$. Let $X^\natural = \sum_{i\in[n]} \omega_i A_i$ where $\omega_i$s are i.i.d. standard Gaussian random variables.  Observe that under the assumption \eqref{eq:xi-relative-profile} we have $M_X(u,p)\le h^p(p) M_{X^\natural}(u,p)$ for every $u\in \mbb R^{d_1\times d_2}$. Therefore, we have
    \begin{align*}
        \left(\E_0\left(M_X(U,p)\right)\right)^{1/p} & \le h(p) \left(\E_0\left(M_{X^\natural}(U,p)\right)\right)^{1/p}              \\
                                                     & = h(p) \left(\E\left(\left|\inp{U,X^\natural}\right|^p\right)\right)^{1/p}\,.
    \end{align*}
    Let $A_{[n]}\in \mbb R^{d_1\times d_2\times n}$ be the tensor formed by stacking $A_i$s along a third axis. With $G_{3} = \bmx{\omega_1, &\dotsb, & \omega_n}^\T$ and $p\ge 2$, applying the bound on the moments of Gaussian chaoses due to \citet[Theorem 1]{Lat06} (see also \citep{Leh10}) to the Gaussian chaos $\inp{U,X^\natural} = \kappa_0\inp{G_1\otimes G_2 \otimes G_3,A_{[n]}}$ yields
    \begin{align*}
        \left(\E\left(\left|\inp{U,X^\natural}\right|^p\right)\right)^{1/p} & \lesssim \kappa_0 \left( \sigma_* p^{3/2} + (\sigma + \upsilon) p + \sigma_\diamond p^{1/2}\right)\,.
    \end{align*}
    Therefore, we have
    \begin{align}
        \left(\E_0\left(M_X(U,p)\right)\right)^{1/p} & \lesssim \kappa_0  h(p) \left( \sigma_* p^{3/2} + (\sigma + \upsilon) p + \sigma_\diamond p^{1/2}\right) \,. \label{eq:tensor-moment-bound}
    \end{align}

    Furthermore, for $x\in \mbb R^{d_1\times d_2}$ with $\norm{x}_{\mr{op}}=1$, let $(v_i\in \mbb R^{d_i})_{i\in [2]}$ be the top singular vectors of $x$ for which $\inp{v_1\otimes v_2,x} = 1$. Also, for $i\in[2]$ let $\tilde{G}_i = (I_{d_i} - v_i v_i^\T)G_i + \tilde{g}_i v_i$ where $I_{d_i}$ is the $d_i\times d_i$ identity matrix and $\tilde{g}_i$ is a standard Gaussian random variable independent of everything else. Clearly, $\tilde{G}_i\sim \mr{Normal}(0,I_{d_i})$ and $G_i - \tilde{G}_i \laweq \sqrt{2} g_i v_i$ where $\laweq$ denotes equivalence in distribution and $g_i\sim \mr{Normal}(0,1)$. Let $Z_1 = U$, $\tilde{Z}_1 = \kappa_0 \tilde{G}_1\otimes G_2$,
    \begin{align*}
        Z_2         & = \kappa_0 (G_1 - \tilde{G}_1) \otimes G_2 \,,           \\
        \tilde{Z}_2 & =  \kappa_0 (G_1 - \tilde{G}_1) \otimes \tilde{G}_{2}\,,
    \end{align*}
    and
    \begin{align*}
        Z_3 & = \kappa_0 (G_1 - \tilde{G}_1) \otimes (G_{2} - \tilde{G}_{2}) \\
            & \laweq 2\kappa_0 g_1g_2\, v_1\otimes v_2\,,
    \end{align*}
    and observe that $Z_{i+1} = Z_i - \tilde{Z}_i$. For each $i\in [2]$ we have $Z_i \laweq \tilde{Z}_i$ and thereby
    \begin{align*}
        2\P\left(\left|\inp{Z_i,x}\right| \ge 1\right) & = \P\left(\left|\inp{Z_i,x}\right| \ge 1\right) + \P\left(\left|\inp{\tilde{Z}_i,x}
        \right| \ge 1\right)                                                                                                                 \\
                                                       & \ge \P\left(\left|\inp{Z_i,x}\right| + \left|\inp{\tilde{Z}_i,x}
        \right| \ge 2\right)                                                                                                                 \\
                                                       & \ge \P\left(\left|\inp{Z_{i+1},x}\right| \ge 2\right)\,.
    \end{align*}
    Therefore, we obtain
    \begin{align*}
        \P\left(\left|\inp{Z_1,x}\right| \ge 1\right) & \ge 4^{-1}\P\left(\left|\inp{Z_{3},x}\right| \ge 4 \right)\,.
    \end{align*}
    Recalling the definitions of $Z_1$ and $Z_{3}$, and the fact that $\inp{v_1\otimes v_2,x} =\norm{x}_{\mr{op}} = 1$ by choice, it follows that
    \begin{align*}
        \P_0\left(\left|\inp{U,x}\right| \ge 1\right) & \ge 4^{-1}\P\left(\kappa_0 \left|g_1 g_2\right|\ge 2\right)                         \\
                                                      & \ge 4^{-1}\P\left(|g_i|\ge \sqrt{2}\kappa_0^{-1/2}\ \text{for all}\ i\in [2]\right) \\
                                                      & = \prod_{i\in [2]}\Phi(-\sqrt{2}\kappa_0^{-1/2})                                    \\
                                                      & \ge 4^{-2}e^{-8/(3\kappa_0)}  \,,
    \end{align*}
    where again we used \eqref{eq:Gaussian-CDF-bound} in the fourth line. Consequently $\nu(\P_0) \le 16 e^{8/(3\kappa_0)}$. In view of \eqref{eq:tensor-moment-bound}, it follows from \ref{lem:main-lemma} that with probability at least $1-e^{-t}$ we have
    \begin{align*}
        \norm{X}_{\mr{op}} & \lesssim \inf_{p\ge 2,\, \kappa_0 >0} e^{t/p} 16^{1/p}e^{8/(3p\kappa_0)}\kappa_0 h(p) \left( \sigma_* p^{3/2} + (\sigma + \upsilon) p + \sigma_\diamond p^{1/2}\right)\,.
    \end{align*}
    Simple calculus shows that the choice of $\kappa_0 = 8/(3p)$ is optimal. Evaluating the objective at this value yields \eqref{eq:matrix-series-tail-bound}.
\end{proof}
\section{Further Abstraction via Coupling}\label{sec:coupling-formulation}
The argument behind \ref{lem:main-lemma}, due to its simplicity, can be applied in a much broader scope. Instead of bounding norms of random vectors or matrices, in principle, we can derive tail bounds or moment bounds for rather complicated functions of a random variable, e.g., supremum over certain class of functionals of a random variable.
The facts that $\norm{X}$ is compared with a linear form $\inp{U,X}$, and $U$ is taken to be independent of $X$ are other implicit limitations of \ref{lem:main-lemma}. The following lemma avoids such limitations and provides a more abstract generalization of \ref{lem:main-lemma} by expressing the bounds in terms of a coupling with $X$. In particular, it provides an \emph{exact} variational formulation of the moments of interest.

\begin{s-lem}\label{lem:coupling-abstraction}
Let $F\st \mbb R^d \to \mbb R$ be a function that has a finite moment of order $p$ with respect to the random variable $X\in \mbb R^d$. Consider a coupling $(X,Y)$ where the random variable $Y\in \mbb R$, which may depend on $X$, is such that $(Y)_+$ has a finite moment. With $\mr{ess\,supp}(X)$ denoting the essential support of $X$, define
\begin{align}
	\nu\left(\P\left(Y\,\mid X\right)\right) = \nu_F\left(\P\left(Y\,\mid X\right)\right) & \defeq \sup_{x\in \mr{ess\,supp}(X)} \frac{1}{\P\left(Y\ge F(X)\,\mid X= x\right)} \label{eq:nu_F}
\end{align}
Then, for any $p\ge 1$ we have
\begin{align}
	\left(\E\left(\left(F(X)\right)_+^p\right)\right)^{1/p} & = \inf_{\P\left(Y\,\mid X\right)} \left(\E\left((Y)^p_+\right)\right)^{1/p} \nu^{1/p}\left(\P\left(Y\,\mid X\right)\right)\,.\label{eq:coupling-moments-variational}
\end{align}
Furthermore, for any $t\ge 0$ with probability at least $1-e^{-t}$ we have
\begin{align}
	F\left(X\right) & \le \inf_{\substack{\P\left(Y\mid X\right) \\ b\in \mbb R\\ p\ge 1
		                      }}b + e^{t/p}\left(\E\left((Y-b)_+^p\right)\right)^{1/p} \nu^{1/p}\left(\P\left(Y\,\mid X\right)\right)\,. \label{eq:coupling-tail-bound}
\end{align}
More precisely, for any $s\in \mbb R$ we have the following inequality
\begin{align}
	\P\left(F(X)\ge s\right) & \le \nu\left(\P\left(Y\st[\mid]X\right)\right)\P\left(Y \ge s\right)\,, \label{eq:tail-conversion}
\end{align}
using which the quantiles of $F(X)$ can be bounded in terms of the quantiles of $Y$.
\end{s-lem}

\begin{proof}
	By Markov's inequality we have
	\begin{align*}
		\P\left(Y\ge F(X)\,\mid X \right) & \le \P\left((Y-b)_+\ge \left(F(X)-b\right)_+\,\mid X \right)            \\
		                                  & \le \frac{\E\left((Y-b)_+^p\,\mid X\right)}{\left(F(X)-b\right)_+^p}\,,
	\end{align*}
	and thus
	\begin{align*}
		\left(F(X)-b\right)_+^p & \le \frac{\E\left((Y-b)_+^p\,\mid X\right)}{\P\left(Y\ge F(X)\,\mid X \right)}                                   \\
		                        & \le \frac{\E\left((Y-b)_+^p\,\mid X\right)}{\inf_{x\in \mr{ess\, supp}(X)}\P\left(Y\ge F(X)\,\mid X = x \right)} \\
		                        & = \E\left((Y-b)_+^p\,\mid X\right)\nu\left(\P\left(Y\,\mid X\right)\right)\,.
	\end{align*}
	Taking the expectation with respect to $X$ we obtain
	\begin{align}
		\E\left(\left(F(X)-b\right)_+^p\right) & \le \E\left((Y-b)_+^p\right) \nu\left(\P\left(Y\,\mid X\right)\right)\,.\label{eq:coupling-truncated-moment-bound}
	\end{align}
	Then, \eqref{eq:coupling-moments-variational} follows by setting $b=0$ and taking the infimum with respect to $\P(Y\,\mid X)$ and observing that for $Y = F(X)$ we have $\nu\left(\P(Y\,\mid X)\right) = 1$. The tail bound
	\eqref{eq:coupling-tail-bound} follows from \eqref{eq:coupling-truncated-moment-bound}, Markov's inequality, and the inequality $F(X) \le \left(F(X)-b\right)_+ + b$.

	To prove \eqref{eq:tail-conversion}, observe that
	\begin{align*}
		\P\left(Y\ge s\right) & = \E\left(\bbone\left(Y \ge s\right)\right)                                               \\
		                      & \ge \E\E\left(\bbone\left(Y \ge F(X)\right)\bbone\left(F(X)\ge s\right)\st[\mid] X\right) \\
		                      & = \E\left(\P\left(Y\ge F(X)\st[\mid] X\right)\bbone\left(F(X)\ge s\right)\right)          \\
		                      & \ge \P\left(F(X)\ge s\right)/\nu\left(\P\left(Y\st[\mid] X\right)\right)\,,
	\end{align*}
	multiplying both sides by $\nu\left(\P\left(Y\st[\mid] X\right)\right)$ yields the claimed inequality.
\end{proof}

While the choice of $Y = F(X)$ attains the infimum in \eqref{eq:coupling-moments-variational}, interesting situations occur when a nearly optimal bound can be achieved by restricted choices of the coupling $(X,Y)$. An obvious practical requirement for suitable couplings is that $\E\left(Y^p\right)$ and $\nu\left(\P\left(Y\,\mid X\right)\right)$ must be easy-to-approximate. For example, \ref{lem:main-lemma}, as a special case of \ref{lem:coupling-abstraction} with $F\left(X\right)= \norm{X}$, only considers $Y = \left|\inp{U,X}\right|$ where $U\sim \P_0$ is independent of $X$. A more refined choice is a \emph{Gibbs measure} with a suitable potential function as the conditional law of $U$ given $X$. It is also worth mentioning that \ref{lem:coupling-abstraction} uses a slightly refined definition of $\nu\left(\P_0\right)$ of \ref{lem:main-lemma} in cases where $X/\norm{X}$ is (almost surely) supported on a \emph{proper} subset of $\partial \mc B$.

We can also use a similar argument to bound the expectation of supremum over functions of $X$ that belong to a prescribed class of functions. The following proposition provides such a bound.
\begin{prop}\label{prop:pseudo-PAC-Bayes}
	Given an index set $\Theta$ let $\left(F_\theta\st{}\mbb R^d \to \mbb R\right)_{\theta\in \Theta}$ be a set of functions that each have a finite moment with respect to the random variable $X\in \mbb R^d$. With $(X,Y)\in \mbb R^d\times \mbb R$ being a coupling induced by the conditional distribution $\P\left(Y\st[\mid] X\right)$ such that $(Y)_+$ has a finite moment, and recalling the definition of $\nu_{F}(\cdot)$ in \eqref{eq:nu_F}, we have
	\begin{align}
		\E\left(\sup_{\theta \in \Theta} F_\theta(X)\right) & \le \inf_{p\ge 1,\,\P\left(Y\st[\mid] X \right)} \left(\E\left((Y)_+^p\right)\right)^{1/p}\sup_{\theta\in \Theta}\nu_{F_\theta}^{1/p}(\P\left(Y\st[\mid] X\right))\,. \label{eq:pseudo-PAC-Bayes}
	\end{align}
\end{prop}
\begin{proof}
	Following the same argument as in the proof of \ref{lem:coupling-abstraction}, for all $\theta \in \Theta$ we have
	\begin{align*}
		\left(F_\theta(X)\right)_+^p & \le \E\left(\left(Y\right)_+^p\st[\mid] X\right) \nu_{F_\theta}\left(\P\left(Y\st[\mid]X\right)\right)\,.
	\end{align*}
	Taking the $p$-th root on both sides and then maximizing with respect to $\theta \in \Theta$ we get
	\begin{align*}
		\sup_{\theta \in \Theta }\left(F_\theta(X)\right)_+ & \le \left(\E\left(\left(Y\right)_+^p\st[\mid] X\right)\right)^{1/p} \sup_{\theta\in \Theta}\nu^{1/p}_{F_\theta}\left(\P\left(Y\st[\mid]X\right)\right)\,.
	\end{align*}
	Then, \eqref{eq:pseudo-PAC-Bayes} follows by taking the expectation with respect to $X$, applying Jensen's inequality for the concave function $z\mapsto z^{1/p}$, and using the fact that $\sup_{\theta\in \Theta} (F_\theta(x))_+ \ge \sup_{\theta\in \Theta} F_\theta(x)$ for all $x\in \mbb R^d$.
\end{proof}

The following corollary, considers a situation where the bound above leads to a more recognizable PAC-Bayesian-style bound in which the corresponding $\nu_{F_\theta}\left(\P\left(Y\st[\mid] X\right)\right)$ is approximated by a certain combination of the KL divergence and the R\'{e}nyi divergence (see, e.g., \citep[Section 3.2.3]{RS13}, \citep[Section 7.12]{PW24}, \citep{Ana17}, and \citep{vEH14} for its properties and relation with the KL divergence) between the prior and posterior measures. Recall that for two probability measures $\mu$ and $\mu_{\mr{ref}}$ with $\mu\ll \mu_{\mr{ref}}$ the R\'{e}nyi divergence of order $\alpha\in \mbb R_{>0} \backslash\{1\}$ of $\mu$ relative to $\mu_{\mr{ref}}$ is
\begin{align*}
	D_{\alpha}\left(\mu,\mu_{\mr{ref}}\right) & = \frac{1}{\alpha-1}\log \E_{\mu_{\mr{ref}}}\left(\left(\frac{\d\mu}{\d\mu_{\mr{ref}}}\right)^\alpha\right)\,.
\end{align*}
\begin{cor}\label{cor:Renyi-divergence}
	Let $X\sim \mu_X$ be a random variable in $\mbb R^d$, $\mc M$ be a set of probability measures on some sample space $\mc Z$, and $f\st{}\mbb R^d\times \mc Z \to \mbb R$ be a function that for each $\mu \in \mc M$ has a finite $L^p(\mu_X \times \mu)$ norm for some $p\ge 1$. Denote the median\footnote{We chose the median instead of any other fixed quantile for simplicity and to avoid introducing extra parameters.} of $f(x,Z)$ with respect to $Z\sim \mu$ by
	\begin{align*}
		F_{\mu}(x) = \med_{Z\sim \mu}\left(f(x,Z)\right)\,.
	\end{align*}
	Then, for any reference measure $\mu_{\mr{ref}}\in \mc M$ and $\alpha>1$ we have
	\begin{align*}
		\E_{\mu_X}\left(\sup_{\mu \in \mc M\st{}\mu \ll \mu_{\mr{ref}}}\hspace{-2ex}F_{\mu}(X)\right) & \le \inf_{p\ge 1} \left(\left(\E_{\mu_X \times \mu_{\mr{ref}}}\left(\left(f(X,Z)\right)_+^p\right)\right)^{1/p} \sup_{\mu \in \mc M\st{}\mu \ll \mu_{\mr{ref}}}\hspace{-2ex}e^{\hat{D}_{\alpha,\mr{KL}}(\mu,\mu_{\mr{ref}})/p}\right)\,,
	\end{align*}
	where
	\begin{align*}
		\hat{D}_{\alpha,\mr{KL}}(\mu,\mu_{\mr{ref}}) & = \min\left\{D_\alpha(\mu,\mu_{\mr{ref}}) + \frac{\alpha}{\alpha - 1}\log(2),\,2D_{\mr{KL}}(\mu,\mu_{\mr{ref}}) + 2\log(2) \right\}\,.
	\end{align*}
\end{cor}
\begin{proof}
	For any $\mu_{\mr{ref}}$-measurable event $\mc E$, it follows from the definition of the R\'{e}nyi divergence for $\alpha>1$ and H\"{o}lder's inequality that
	\begin{align*}
		\mu(\mc E) & = \E_{\mu_{\mr{ref}}}\left(\frac{\d\mu}{\d\mu_{\mr{ref}}}\cdot\bbone_{\mc E}\right)                                                                                                                \\
		           & \le \left(\E_{\mu_{\mr{ref}}} \left(\left(\frac{\d \mu}{\d \mu_{\mr{ref}}}\right)^\alpha\right)\right)^{1/\alpha} \left(\E_{\mu_{\mr{ref}}}\left(\bbone_{\mc E}\right)\right)^{(\alpha -1)/\alpha} \\
		           & = e^{(\alpha -1)D_\alpha(\mu,\mu_{\mr{ref}})/\alpha}\left(\mu_{\mr{ref}}(\mc E)\right)^{(\alpha -1)/\alpha}\,.
	\end{align*}
	In particular, for  $\mc E = {\mc E}_x = \{f(x,Z)\ge F_{\mu}(x)\}$, we get
	\begin{align}
		\P_{Z\sim \mu_{\mr{ref}}}\left(f(x,Z)\ge F_{\mu}(x)\right) & \ge \left(\P_{Z\sim \mu}\left(f(x,Z)\ge F_{\mu}(x)\right)\right)^{\alpha/(\alpha-1)} e^{-D_{\alpha}(\mu,\mu_\mr{ref})} \nonumber \\
		                                                           & \ge 2^{-\alpha/(\alpha-1)}e^{-D_{\alpha }(\mu,\mu_{\mr{ref}})}\,, \label{eq:lower-bound-by-Renyi}
	\end{align}
	where the second line follows from the definition of $F_{\mu}(\cdot)$.

	We can obtain an alternative lower bound for $\mu_{\mr{ref}}(\mc E) = \P_{Z\sim \mu_{\mr{ref}}}\left(f(x,Z)\ge F_{\mu}(x)\right)$ in terms of the KL divergence between $\mu$ and $\mu_{\mr{ref}}$. Recalling the choice of $\mc E$, it follows from the \emph{data processing inequality} (see, e.g., \citep[Theorem 7.4]{PW24}) that
	\begin{align*}
		D_{\mr{KL}}(\mu,\mu_{\mr{ref}}) & \ge D_{\mr{KL}}(\mr{Bernoulli}(\mu(\mc E)), \mr{Bernoulli}(\mu_{\mr{ref}}(\mc E)))            \\
		                                & =  -\frac{1}{2}\log\left(4\mu_{\mr{ref}}(\mc E)\left(1-\mu_{\mr{ref}}(\mc E)\right)\right)\,,
	\end{align*}
	or equivalently
	\begin{align*}
		\left|1-2\mu_{\mr{ref}}(\mc E)\right| & \le \sqrt{1-e^{-2D_{\mr{KL}}(\mu,\mu_{\mr{ref}})}}\,,
	\end{align*}
	from which we can extract the one-sided inequality
	\begin{align}
		\P_{Z\sim \mu_{\mr{ref}}}\left(f(x,Z)\ge F_{\mu}(x)\right) & \ge \frac{1}{2}\left(1-\sqrt{1-e^{-2D_{\mr{KL}}(\mu,\mu_{\mr{ref}})}}\right) \nonumber \\
		                                                           & \ge \frac{1}{4}e^{-2D_{\mr{KL}}(\mu,\mu_{\mr{ref}})}\,. \label{eq:lower-bound-by-KL}
	\end{align}

	Combining \eqref{eq:lower-bound-by-Renyi} and \eqref{eq:lower-bound-by-KL} we deduce
	\begin{align*}
		\nu_{F_\mu}\left(\P\left(Y\st[\mid]X\right)\right) \le \min\left\{2^{\alpha/(\alpha-1)}e^{D_{\alpha }(\mu,\mu_{\mr{ref}})},\,4e^{2D_{\mr{KL}}(\mu,\mu_{\mr{ref}})} \right\}= e ^{\hat{D}_{\alpha,\mr{KL}}(\mu,\mu_{\mr{ref}})}\,.
	\end{align*}
	The claim follows by applying \ref{prop:pseudo-PAC-Bayes} with $Y = f(X,Z)$ for $Z\sim \mu_{\mr{ref}}$ and treating $\mu$ as the index for $F_{\mu}(\cdot)$.
\end{proof}

Below we provide a simple application of \ref{prop:pseudo-PAC-Bayes}.
\begin{exmp}\label{exmp:L-inf-non-iid-Gussians}
	Given $\sigma_1\ge \dotsc\ge \sigma_d>0$, for $i=1,\dotsc,d$ let $X_i\sim\mr{Normal}(0,\sigma_i^2)$ be independent centered Gaussian random variables. Our goal is to bound $\E\left(\norm{X}_\infty\right)$ where $X = \bmx{X_i}_{i=1}^d\in \mbb R^d$. Let $\Theta = [d]$ and for $\theta \in \Theta$ define $F_\theta\st{}\mbb{R}^d\to \mbb R$ as $F_\theta(x) = |x_\theta| - b$ for some parameter $b\ge 0$. Furthermore, take $J$ to be uniformly distributed over $[d]$ independent of everything else and let $Y = |X_J| - b$. We have $\nu_{F_\theta}(\P\left(Y\mid\, X\right))\le d$ and it follows from \ref{prop:pseudo-PAC-Bayes} that
	\begin{align*}
		\E\left(\max_{i\in [d]} |X_i| - b\right) & \le d^{1/p} \left(\E\left((Y)_+^p\right)\right)^{1/p}                        \\
		                                         & = \left(\sum_{j=1}^d \E\left(\left(|X_j|-b\right)_+^p\right)\right)^{1/p}\,.
	\end{align*}
	Each of the summands on the right-hand side can be bounded as
	\begin{align*}
		\E\left(\left(|X_j|-b\right)_+^p\right) & = \int_0^\infty \P\left(|X_j| \ge s + b\right) p s^{p-1}\d s                                 \\
		                                        & \le \int_0^\infty \inf_{\lambda\ge 0} 2e^{\lambda^2 \sigma_i^2/2-\lambda(s+b)} p s^{p-1}\d s \\
		                                        & = \int_0^\infty 2e^{-(s+b)^2/(2\sigma_i^2)} p s^{p-1}\d s                                    \\
		                                        & \le 2\sigma_i^p e^{-b^2/(2\sigma_i^2)}\left(\frac{p}{2}\right)!\,.
	\end{align*}
	Combining the derived inequalities and optimizing with respect to $p\ge 1$ we obtain
	\begin{align*}
		\E\left(\norm{X}_\infty\right) & \le b + \inf_{p\ge 1} \sqrt{2p}\left(\sum_{i=1}^d \sigma_i^p  e^{-b^2/(2\sigma_i^2)}\right) ^{1/p}\,.
	\end{align*}
	Since the optimal choice of $b\ge 0$  and $p\ge 1$ cannot be expressed explicitly as a function of the $\sigma_i$s, we resort to approximation. Let
	\begin{align*}
		\sigma_* & = \max_{i\in[d]} \sigma_i \sqrt{\log(i+1)}\,,
	\end{align*}
	and choose $b = c\sigma_*$ for some coefficient $c>0$ to obtain
	\begin{align*}
		\E\left(\norm{X}_\infty\right) & \le c\sigma_* + \inf_{p\ge 1} \sqrt{2p}\sigma_1\left(\sum_{i=1}^d \left(i+1\right)^{-c^2/2}\right) ^{1/p}\,.
	\end{align*}
	In particular, for $c=2$ we have
	\begin{align*}
		\E\left(\norm{X}_\infty\right) & < 2\sigma_* + \sqrt{2}\sigma_1\,,
	\end{align*}
	which, in view of the fact that $\sigma_1 \le \sigma_*/\sqrt{\log 2}$, coincides with the tight estimate of $\E\left(\norm{X}_\infty\right)$ (see, e.g., \citep[Section 2.1]{vHan17}). Applying the Gaussian concentration inequality we also obtain the moment and tail bounds for $\norm{X}_\infty$.
\end{exmp}
\appendix
\crefalias{section}{appendix}
\section{Recovering \ref{exmp:polyhedral-Gaussians} and \eqref{eq:subGauss-Euclidean} via \ref{thm:gaussian-pushforward}}\label{apx:examples-via-pushforward}
We can recover the bound \eqref{eq:polyhedral-Gaussian-simplified} in \ref{exmp:polyhedral-Gaussians} from \ref{thm:gaussian-pushforward} up to a factor of order $\sqrt{\log c_k}$ that can be ignored if $c_k$ is not large relative to $\pi_k^{-1}$. The parameters of the problem in this example are $\eta_1=\sqrt{2}$, $\eta_2 = 0$, and $\Sigma = I$. By taking $\kappa_0 \to \infty$, we have $f_0(x) = \rho_0 \norm{x}$ which corresponds to the choice
\begin{align*}
	U = \nabla f_0 (G) & = \argmax_{u\in \rho_0 \{\pm u_1,\dotsc, \pm u_N\}} \inp{u, G}\,,
\end{align*}
as considered in \ref{exmp:polyhedral-Gaussians}. Recalling the definitions \eqref{eq:exmp1-c_k} and \eqref{eq:exmp1-pi_k}, we focus on the case $\rho_0 = 2c_k$ for some  $k\in[N]$. Because
\begin{align*}
	\norm{\Sigma}_{\square} & = \sup_{u\in \mc B_*} u^\T \Sigma u = \sup_{u\in \mc B_*} \norm{u}_2^2  = \max_{i\in[N]} \norm{u_i}_2^2\,,
\end{align*}
we have
\begin{align*}
	\Omega_{\Sigma}(p,\infty,2c_k) & \le \eta_1 \sqrt{p} \rho_0 \norm{\Sigma}_\square^{1/2}= 2\sqrt{2p} c_k \max_{i\in [N]}\norm{u_i}_2^2\,.
\end{align*}
Next, we bound $\bar{\nu}_{\Sigma}(\infty,2c_k)\le \rho_0/\ubar{\lambda}_{\Sigma}(\infty,2c_k)$. Let $I_k(x) =\left\{i\in[N]\st c_k|\inp{u_i,x}| \ge 1 \right\}$. We have
\begin{align*}
	\ubar{\lambda}_{\Sigma}(\infty,2c_k) & =  \inf_{x\in \partial B} \sum_{i\in [N]} 2\P_0\left(U = 2c_k u_i\right) \left(2c_k\left|\inp{u_i,x}\right|-1\right)_+ \\
	                                     & \ge \inf_{x\in \partial B} \sum_{i\in I_k(x)} 2\P_0\left(U = 2c_k u_i\right)c_k\left|\inp{u_i,x}\right|                \\
	                                     & \ge 2\pi_k\,,
\end{align*}
using which we obtain the bound
\begin{align*}
	\bar{\nu}_{\Sigma}(\infty, 2c_k) & \le \frac{c_k}{\pi_k} \,.
\end{align*}
Therefore, the bound \eqref{eq:G-pushforward} provided by \ref{thm:gaussian-pushforward} implies
\begin{align*}
	\norm{X} & \le 2\max_{i\in[N]}\norm{u_i}^2_2 \inf_{p\ge 1,\, k\in[N]}e^{t/(2p)} \sqrt{2p}c_k\left(\frac{c_k}{\pi_k}\right)^{1/(2p)}\\
	         & \le \sqrt{8e} \max_{i\in [N]}\norm{u_i}_2^2 \min_{k\in [N]} c_k\max\left\{1,\sqrt{t-\log(\pi_k) + \log(c_k)}\right\}\,.
\end{align*}
Let $k_\*\in [N]$ be the minimizer of $\min_{k\in [N]} c_k\max\left\{1,\sqrt{t-\log \pi_k}\right\}$ in \eqref{eq:polyhedral-Gaussian-simplified}. If we have $\log c_{k_\*} \lesssim t- \log \pi_{k_\*}$, our derivations confirm that, up to constant factors, \ref{thm:gaussian-pushforward} recovers the bound \eqref{eq:polyhedral-Gaussian-simplified} of \ref{exmp:polyhedral-Gaussians}.

Next we show that the tail bound \eqref{eq:subGauss-Euclidean}, which applies when $X$ has sub-Gaussian marginals, can be recovered from \ref{thm:gaussian-pushforward}. We have $\eta_1 =\eta$ and again $\eta_2 = 0$. Also, recall that in this example $\mc B$ and $\mc B_*$ coincide with the unit $\ell_2$-ball. By taking $\rho_0 \to +\infty$ we have $f_0(x) = \kappa_0\norm{x}_{\Sigma^{-1/2}}^2/2$, and thus
\begin{align*}
	U = \nabla f_0(\Sigma^{1/2}G) & = \kappa_0 G\,.
\end{align*}
Furthermore, because $\norm{G}_*$ has a sub-Gaussian tail we have $\lim_{\rho_0 \to +\infty} T_{\geqslant}(\rho_0/\tau \kappa_0) \rho_0 = 0$. Therefore, choosing $\tau=1$ in the definition of $\Omega_{\Sigma}(p,\kappa_0,\rho_0)$ we obtain
\begin{align*}
	\Omega_{\Sigma}(p,\kappa_0,\infty) & = \eta \sqrt{p}\kappa_0 \left(\sqrt{\tr(\Sigma)} + \sqrt{2p\norm{\Sigma}_{\mr{op}}}\right)\,.
\end{align*}
Furthermore, since $\inp{G,x}$ has the standard Gaussian distribution we can write
\begin{align*}
	\ubar{\lambda}_{\Sigma}(\kappa_0,\infty)  = \E\left(\left(\kappa_0\left|\inp{G,x}\right|-1\right)_+\right) & = \frac{2}{\sqrt{2\pi}}\int_{1/\kappa_0}^\infty (\kappa_0 z - 1)e^{-z^2/2}\d z                                            \\
	                                                                                                           & \ge \frac{2}{\sqrt{2\pi}}\int_{1/\kappa_0}^\infty (\kappa_0 z - 1)e^{-\left((z-\kappa_0^{-1})^2+\kappa_0^{-2}\right)}\d z \\
	                                                                                                           & = \frac{2\kappa_0}{\sqrt{2\pi}}e^{-\kappa_0^{-2}}\int_0^\infty se^{-s^2}\d s                                              \\
	                                                                                                           & = \frac{\kappa_0}{\sqrt{2\pi}}e^{-\kappa_0^{-2}}\,
\end{align*}
where we used the inequality $z^2/2 \le (z-\kappa_0^{-1})^2 + \kappa_0^{-2}$ on the second line, and the change of variable $s = z - \kappa_0^{-1}$ on the fourth line. Applying the inequality above in the definition of $\bar{\nu}_{\Sigma}(\kappa_0,\infty)$ yields
\begin{align*}
	\bar{\nu}_\Sigma(\kappa_0, \infty) & \le 2\pi e^{1/\kappa_0^2}+1< 8 e^{1/\kappa_0^2}\,.
\end{align*}
Therefore, the bound \eqref{eq:G-pushforward} of \ref{thm:gaussian-pushforward} implies
\begin{align*}
	\norm{X} = \norm{X}_2 & \le \inf_{p\ge 1,\, \kappa_0> 0 } e^{t/(2p)}\eta\sqrt{p}\kappa_0 \left(\sqrt{\tr(\Sigma)} + \sqrt{2p\norm{\Sigma}_{\mr{op}}}\right)8^{1/(2p)}e^{1/(2p\kappa_0^2)}\\
	                      & = \sqrt{e}\eta \inf_{p\ge 1 } e^{(t+\log 8)/(2p)}\left(\sqrt{\tr(\Sigma)} + \sqrt{2p\norm{\Sigma}_{\mr{op}}}\right) \\
	                      & \le 4\sqrt{e}\eta \left(3\tr(\Sigma) + \frac{2}{3\log 2}\norm{\Sigma}_{\mr{op}}t\right)^{1/2}\,,
\end{align*}
where the third line follows by evaluating the argument of the infimum at $p = 1 + t/\log 8$, the Cauchy--Schwarz inequality, and the fact that $\norm{\Sigma}_{\mr{op}} \le \tr(\Sigma)$. Therefore, up to the constant factors, we have recovered the bound \eqref{eq:subGauss-Euclidean} through \ref{thm:gaussian-pushforward}.

\section{Supporting Lemmas}\label{apx:supporting-lemmas}
\begin{lem}\label{lem:moments-of-quadratic-form}
	Let $B\in\mbb R^{\ell\times \ell}$ be a symmetric matrix and $G\in \mbb R^\ell$ be a standard Gaussian random variable. Then, for $p\ge 1$ we have
	\begin{align}
		\E\left(|G^\T B G|^p\right) & \le \norm{B}_*^p \prod_{i=1}^{\lfloor p\rfloor}\left(1 + \frac{2(p-i)\norm{B}_{\mr{op}}}{\norm{B}_*}\right)\label{eq:Gaussian-moment-inequality}\,,
	\end{align}
	where $\norm{B}_*$ denotes the nuclear norm of $B$. The inequality above can be simplified to
	\begin{align*}
		\E\left(|G^\T B G|^p\right) & \le \left(\norm{B}_* + p\norm{B}_{\mr{op}}\right)^p\,.
	\end{align*}
\end{lem}
\begin{proof}
	With $B_+$ and $B_-$ respectively denoting the positive semidefinite parts of $B$ and $-B$, let $B_\natural = B_+ + B_-$ which is a positive semidefinite matrix itself. It follows from the triangle inequality that
	\begin{align*}
		|G^\T B G| & \le G^\T B_\natural G\,.
	\end{align*}
	Using the Gaussian integration by parts (or Stein's lemma) we have
	\begin{align*}
		\E\left(|G^\T B G|^p\right) & \le \E\left(\left(G^\T B_\natural G\right)^p\right)                                                                                                                     \\
		                            & = \tr(B_\natural)\E\left(\left(G^\T B_\natural G\right)^{p-1}\right) + 2(p-1)\E\left(\left(G^\T B_\natural^2 G\right)\left(G^\T B_\natural G\right)^{p-2}\right)        \\
		                            & \le \left(\tr(B_\natural) + 2(p-1)\norm{B_\natural}_{\mr{op}}\right)\E\left(\left(G^\T B_\natural G\right)^{p-1}\right)                                                 \\
		                            & \le \prod_{i=1}^{\lfloor p\rfloor}\left(\tr(B_\natural) + 2(p-i)\norm{B_\natural}_{\mr{op}}\right)\E\left(\left(G^\T B_\natural G\right)^{p-\lfloor p\rfloor}\right)\,,
	\end{align*}
	where the third line follows from the fact that $B_\natural^2 \preceq \norm{B_\natural}_{\mr{op}}B_\natural$, and the fourth line holds by recursion. Since $p-\lfloor p \rfloor \in [0,1)$  and $z\mapsto z^{p-\lfloor p \rfloor}$ is concave over $z\ge 0$, Jensen's inequality guarantees that
	\begin{align*}
		\E\left(\left(G^\T B_\natural G\right)^{p-\lfloor p\rfloor}\right) & \le \left(\tr(B_\natural)\right)^{p-\lfloor p \rfloor}\,.
	\end{align*}
	Combining the derived inequalities, we obtain
	\begin{align*}
		\E\left(|G^\T B G|^p\right) & \le \left(\tr(B_\natural)\right)^p \prod_{i=1}^{\lfloor p\rfloor}\left(1 +\frac{2(p-i)\norm{B_\natural}_{\mr{op}}}{\tr(B_\natural)}\right)
	\end{align*}
	The result follows by observing that $\norm{B_\natural}_{\mr{op}} = \norm{B}_{\mr{op}}$ and $\tr(B_\natural) = \norm{B}_*$.

	To obtain the simplified version of \eqref{eq:Gaussian-moment-inequality}, we can apply the AM--GM inequality to the right-hand side of \eqref{eq:Gaussian-moment-inequality} which yields
	\begin{align*}
		\E\left(|G^\T B G|^p\right) & = \norm{B}_*^p\left(1+(2p-1-\lfloor p\rfloor)\frac{\norm{B}_{\mr{op}}}{\norm{B}_*}\right)^{\lfloor p \rfloor} \\
		                            & \le \left(\norm{B}_*+p\norm{B}_{\mr{op}}\right)^p\,,
	\end{align*}
	as desired.
\end{proof}

\begin{lem}[Second Moment Method]\label{lem:PZ}
	Let $\xi\in \mbb R$ be a random variable. For any $s\in \mbb R$ we have
	\begin{align*}
		\P(\xi \ge s) & \ge \frac{\left(\E\left(\xi - s\right)_+\right)^2}{\E\left((\xi - s)_+^2\right)}\,,
	\end{align*}
	where $(x)_+\defeq \max\{x,0\}$. In particular, if $\E \xi \ge 0$, then for any $\rho\in (0,1)$ we have
	\begin{align*}
		\P(\xi \ge \rho \E \xi) & \ge \frac{(1-\rho)^2(\E \xi)^2}{\var(\xi) + (1-\rho)^2(\E \xi)^2}\,.
	\end{align*}
\end{lem}
\begin{proof}
	By the Cauchy--Schwarz inequality we have
	\begin{align*}
		\P(\xi \ge s)\E\left( (\xi - s)_+^2\right) & = \E\left(\bbone(\xi \ge s)\right)\E\left((\xi-s)_+^2\right) \\
		                                           & \ge \left(\E\left((\xi-s)_+\right)\right)^2\,,
	\end{align*}
	which yields the main inequality. If we further know that $\E\xi\ge 0$, then choosing $s = \rho\E \xi$ yields
	\begin{align*}
		\P(\xi \ge \rho \E \xi) & \ge \frac{\left(\E\left(\xi - \rho\E \xi \right)_+\right)^2}{\E\left((\xi - \rho \E \xi)^2\right)} \\
		                        & \ge \frac{(1-\rho)^2(\E \xi)^2}{\var(\xi) + (1-\rho)^2(\E \xi)^2}\,,
	\end{align*}
	where the second line follows by applying the Jensen's inequality in the numerator and using the assumption $\E \xi \ge 0$.
\end{proof}

\begin{lem}
	\label{lem:gradient-split-bound}
	Recall the gradient of $f_0$ given by \eqref{eq:gradient-as-a-projection}. For every $x\in \mbb R^d$ and $\tau\in (0,1]$, we have
	\begin{align*}
		\norm{\nabla f_0(x)}_{\Sigma} & \le \norm{\Sigma}_{\mr{op}}^{1/4} (1-\tau)\kappa_0\norm{\Sigma^{-1/2}x}_{\Sigma^{1/2}} + \tau \kappa_0 \norm{x}_2\,.
	\end{align*}
\end{lem}
\begin{proof}
	We can view \eqref{eq:gradient-as-a-projection} as a projection (in a Hilbert space with norm $\norm{\cdot}_{\Sigma^{1/2}}$). Therefore, using the fact that projections onto closed convex sets are \emph{firmly nonexpansive} \citep[Propostion 4.8]{BC17} and thus define contractions, for all $x,y\in \mbb R^d$ we have
	\begin{align}
		\norm{\nabla f_0(x) - \nabla f_0(y)}_{\Sigma^{1/2}} & \le \norm{\kappa_0 \Sigma^{-1/2}(x-y)}_{\Sigma^{1/2}}\,. \label{eq:gradients-contract}
	\end{align}
	Therefore, if for $\tau\in(0,1]$ we have $\tau \kappa_0\norm{\Sigma^{-1/2}x}_* < \rho_0$ which guarantees $\nabla f_0(\tau x) = \tau \kappa_0 \Sigma^{-1/2}x$, by choosing $y=\tau x$ we can write
	\begin{align*}
		\norm{\nabla f_0(x)}_{\Sigma} & \le \norm{\nabla f_0(x) - \nabla f_0(\tau x)}_{\Sigma} + \norm{\nabla f_0(\tau x)}_{\Sigma}                           \\
		                              & = \norm{\nabla f_0(x) - \nabla f_0(\tau x)}_{\Sigma} + \tau \kappa_0 \norm{x}_2                                       \\
		                              & \le \norm{\Sigma}_{\mr{op}}^{1/4} \norm{\nabla f_0(x) - \nabla f_0(\tau x)}_{\Sigma^{1/2}} + \tau \kappa_0 \norm{x}_2 \\
		                              & \le \norm{\Sigma}_{\mr{op}}^{1/4} (1-\tau)\kappa_0\norm{\Sigma^{-1/2}x}_{\Sigma^{1/2}} + \tau \kappa_0 \norm{x}_2\,,
	\end{align*}
	where used the triangle inequality in the first line, the fact that $\Sigma \preceq \norm{\Sigma}_{\mr{op}}^{1/2}\Sigma^{1/2}$ in the third line, and \eqref{eq:gradients-contract} in the fourth line.
\end{proof}
\printbibliography
\end{document}